\documentclass[a4paper,reqno]{amsart}

\usepackage[utf8]{inputenc}
\usepackage{amssymb}
\usepackage{enumitem}
\usepackage{mathrsfs}
\usepackage{mathtools}
\usepackage{color}
\usepackage[colorlinks=true]{hyperref}
\hypersetup{urlcolor=green,citecolor=cyan}

\setlist[enumerate]{label=\textnormal{(\roman*)}}

\usepackage{environ}
\NewEnviron{rem}{\begin{quote}\small\color{blue}\BODY\end{quote}}

\newtheorem{theorem}{Theorem}[section]

\newtheorem{lemma}[theorem]{Lemma}
\newtheorem{proposition}[theorem]{Proposition}

\theoremstyle{definition}
\newtheorem{definition}[theorem]{Definition}
\newtheorem{remark}[theorem]{Remark}
\numberwithin{equation}{section}

\newcommand{\RR}{\mathbb{R}}

\newcommand{\m}[1]{\mathbb{#1}}
\newcommand{\q}[1]{\mathcal{#1}}

\renewcommand{\le}{\leqslant}
\renewcommand{\ge}{\geqslant}

\DeclareMathOperator{\Span}{\text{Span}}
\DeclareMathOperator{\Mod}{\mathrm{Mod}}

\def\d{{\rm d}}
\def \E{\mathcal{H}}
\def \vp{\varphi_{1}}
\def\vpp{\varphi_{2}}

\def\ba{{\boldsymbol{a}}^{+}}

\begin{document}

\parindent=0pt

\title[Asymptotics of damped NLKG equation]
{Asymptotics of solutions with a compactness property for the nonlinear damped Klein-Gordon equation}
\author[R.~C\^ote]{Raphaël C\^ote}
\address{IRMA UMR 7501, Université de Strasbourg, CNRS, F-67000 Strasbourg, France}
\email{cote@math.unistra.fr}
	
\author[X.~Yuan]{Xu Yuan}
\address{CMLS, \'Ecole polytechnique, CNRS, Institut Polytechnique de Paris, 91128 Palaiseau Cedex, France}
\email{xu.yuan@polytechnique.edu}
	
\subjclass[2010]{35L71 (primary), 35B40, 37K40}

\begin{abstract}
We consider the nonlinear damped Klein-Gordon equation 
\[ \partial_{tt}u+2\alpha\partial_{t}u-\Delta u+u-|u|^{p-1}u=0 \quad \text{on} \ \  [0,\infty)\times \m R^N \]
 with $\alpha>0$, $2 \le N\le 5$ and energy subcritical exponents $p>2$. We study the behavior of solutions for which it is supposed that only one nonlinear object appears asymptotically for large times, at least for a sequence of times.
 
We first prove that the nonlinear object is necessarily a bound state. Next, we show that when the nonlinear object is a non-degenerate state or a degenerate excited state satisfying a simplicity condition, the convergence holds for all positive times, with an exponential or algebraic rate respectively. Last, we provide an example where the solution converges exactly at the rate $t^{-1}$ to the excited state.
 \end{abstract}
 
\maketitle
\section{Introduction}

\subsection{Setting of the problem} 
We consider the nonlinear focusing damped Klein-Gordon equation
\begin{equation}\label{equDKG}
\partial_{tt}u+2\alpha\partial_{t}u-\Delta u+u-f(u)=0\quad (t,x)\in [0,\infty)\times\RR^{N},
\end{equation}
where $f(u)=|u|^{p-1}u$, $\alpha>0$, $2\le N\le 5$ and  the exponent $p$ satisfies
\begin{equation*}
2<p<p^{*}(N)\quad \text{with}\quad  p^{*}(N)=\left\{
\begin{aligned}
	&\quad \infty\quad \quad \text{if}\ N=2,\\
	&\frac{N+2}{N-2}\quad \text{if}\ N=3,4,5.
	\end{aligned}\right.
\end{equation*}

It follows from \cite[Theorem~2.3]{BRS} that the Cauchy problem for~\eqref{equDKG} is locally well-posed in the energy space: for any initial data
$(u_0,v_0)\in H^1(\m R^N)\times L^2(\m R^N)$, there exists a unique (in some class) maximal solution
$u \in \mathscr C([0,T_{\max}),H^1(\m R^N))\cap \mathscr C^1([0,T_{\max}),L^2(\m R^N))$ of \eqref{equDKG}. Moreover, if the maximal time of existence $T_{\max}$ is finite, then
$\lim_{t\uparrow T_{\max}} \| \vec u(t)\|_{H^1 \times L^2}=\infty$.

Setting $F(u)= \frac{1}{p+1}|u|^{p+1}$ and
\[
E(\vec u)=\frac 12 \int_{\RR^N} \big\{ |\nabla u|^2 + u^2 + (\partial_t u)^2
- 2 F(u)\big\} \d x,
\]
for any $H^1\times L^2$ solution $\vec u=(u,\partial_{t} u)$ of \eqref{equDKG}, it holds
\begin{equation}\label{energy}
E(\vec u(t_2))-E(\vec u(t_1)) = -2 \alpha \int_{t_1}^{t_2} \|\partial_t u(t)\|_{L^2}^2 \d t.
\end{equation}

One can easily construct finite time blow-up solutions by adequately truncating a constant in space solution, whose initial data lead to finite time blow-up for the inferred ODE $y'' + 2 \alpha y' + y - f(y)=0$ (and using finite speed of propagation). On the other hand, solutions to \eqref{equDKG} which are globally defined for positive time, that is for which $T_{\max} = +\infty$, are believed to possess much more structure, in the spirit of a soliton resolution: it roughly asserts that any global solution (maybe under a genericity condition) splits for large times into a sum of decoupled rigid nonlinear objects, which should be here stationary solutions, especially in view of decay of energy \eqref{energy}.

\bigskip

Let us first recall  from \cite{Ca} (see also references therein) some features on stationary solution, namely a solution to the elliptic equation
\begin{equation} \label{q}
- \Delta q + q - f(q) =0, \quad q \in H^1(\m R^N).
\end{equation}
We call  the solutions of \eqref{q} \emph{bound states}, and denote $\mathcal B$ the set of bound states:
\[
\mathcal B = \{ q :  q \text{ is a nontrivial solution of } \eqref{q}\}.
\]
Standard elliptic arguments (see \emph{e.g}.~\cite{GT77} or~\cite[Theorem~8.1.1]{Ca}) show that if $q\in \mathcal B$, then $q$ is of class $\q C^2(\m R^N)$ and has exponential decay as $|x|\to +\infty$, as well as its first and second-order derivatives.

Let
\[
\mathcal W(v) = \frac 1 2 \int_{\RR^{N}} \left\{|\nabla v|^2 + v^2 - 2F(v)\right\} \d x,\quad \mbox{for}\ v\in H^{1}.
\]
We call the solutions of \eqref{q} which minimize the functional $\mathcal{W}$ by \emph{ground states}; the set of ground states is denoted by $\mathcal G$
\[
\mathcal G = \{ q_0 \in \q B : \forall q \in \q B, \  \mathcal W(q_0) \le \mathcal W(q) \}.
\]

Ground states are well studied objects. They are unique up to space translation (for rather general nonlinearities): there exists a radial positive function $q_{0}$ of class $\mathscr C^2$, exponentially decreasing, along with its first and second-order derivatives, such that
\[ \mathcal G = \{ q_0 (x-x_0) : x_0\in \m R^N\}. \]
We refer to Berestycki-Lions \cite{BL}, Gidas-Ni-Nirenberg \cite{GNN}, Kwong \cite{K}, Serrin-Tang \cite{ST} (however, a positive bound state may not be a ground state, see \cite{DPG12}). It is well-known (see~\emph{e.g.} Grillakis-Shatah-Strauss~\cite{GSS}) that the ground state $q_{0}$ is unstable in the energy space. This result was also known in the physics literature as the Derrick’s Theorem~\cite{D}.

In dimension $1$, $\mathcal B = \mathcal G$ (due to ODE arguments). In contrast, for any $N \ge 2$, $\mathcal G \subsetneq \mathcal B$: see~\cite[Remark~8.1.16]{Ca}. Functions $q\in \mathcal B\setminus \mathcal G$ are referred to as \emph{excited states}. As a matter of fact, much less is known about excited states.

Here are some references on the construction of excited states. Berestycki-Lions \cite{BL2} showed the existence of infinitely many radial \emph{nodal} (\emph{i.e}. sign changing) solutions (see also \cite{HV94,MPW} and the references therein). For the massless version of equation \eqref{q}, the existence of excited states that are nonradial sign-changing and with arbitrary large energy was first proved in Ding~\cite{Ding} by variational argument. Later, del Pino-Musso-Pacard-Pistoia \cite{DMPP11} constructed more explicit solutions to the massless equation \eqref{q} with a centered soliton crowned with negative spikes (rescaled solitons) at the vertices of a regular polygon of radius 1. Then, following similar general strategy in~\cite{DMPP11}, they constructed sign changing, non radial solutions to \eqref{q} on the sphere $\m S^N$ ($N \ge 4$) whose energy is concentrated along special submanifolds of $\m S^N$ in~\cite{DMPP13}.

\smallskip

We can now go back to \eqref{equDKG}, and recall some previous results related to the long time dynamics of global solutions.

Under some conditions on $N$ and $p$, results in~\cite{F98,LZ} state that for any sequence of time, any global bounded solution of~\eqref{equDKG} converges to a sum of decoupled bound states after extraction of a subsequence of times. Also in \cite{F98}, Feireisl constructed global solutions that behave as sum of an \emph{even} number of ground states (\emph{i.e.} multi-solitons).

In~\cite{BRS}, for dimension $N\geq 2$, Burq, Raugel and Schlag proved the convergence of any global \emph{radial} solution to one (radial) bound state, for the whole sequence of time.

In \cite{CMYZ}, it is given a complete description of $2$-soliton solutions (that is, solutions which, on at least a sequence of time, behave as the sum of two decoupled ground states), in dimension $N\le 5$. Building on the tools developed there, \cite{CMY} gave a complete description of global solutions in dimension $N=1$, that is, the soliton resolution in that case. 

\smallskip

We aim at considering the behavior of solution without conditions on symmetry (like radiality). A complete description seems out of reach, because of the lack of understanding of the dynamics around general excited states, and because the system of centers of mass of the involved bound states may have itself a very intricate dynamics.

\subsection{Main results}

In this paper, we are instead interested in understanding the behavior of solutions to \eqref{equDKG} for which \emph{only one} nonlinear object appears for large times, at least for a sequence of time. More precisely, we define \emph{packed} solutions as follows.

\begin{definition}
 A maximal solution $\vec{u}=(u,\partial_{t}u)\in \mathscr C([0,T_{\rm{max}}),H^{1}\times L^{2})$ of~\eqref{equDKG} is called a \emph{packed solution} if there exist $(W_0,W_1) \in H^{1} \times L^2$, and a time sequence $t_{n}\to T_{\rm{max}}$ and a position sequence $y_n \in \m R^N$ such that 
 \begin{equation} \label{def:packed}
 \lim_{{n}\to\infty}\left\{\|u(t_{n})-W_0(\cdot-y_{n})\|_{H^{1}}+\|\partial_{t}u(t_{n}) - W_1(\cdot-y_{n})\|_{L^{2}}\right\}=0.
 \end{equation}
 We say that $\vec W = (W_0,W_1)$ is a cluster point for $\vec u$ at $(t_n, y_n)_n$.
\end{definition}

Observe that any cluster point $(W_0,W_1)$ is actually a bound state $(q,0)$. More precisely,
the following Proposition holds true.

\begin{proposition}\label{prop:1.1}
Let $\vec{u}=(u,\partial_{t}u)$ be a packed solution of~\eqref{equDKG}. Then $\vec u \in \mathscr C([0,+\infty),H^{1}\times L^{2})$ is globally defined for positive times, and if $(W_0,W_1) \in H^1 \times L^2$ is a cluster point for $\vec u$ at $(t_n, y_n)_n$, then  $W_0=q$ is a bound state of~\eqref{q} and $W_1=0$. Furthermore, the energy is bounded below, $\partial_t u \in L^2([0,+\infty), L^2)$ and for all $t \ge 0$,
\begin{equation} \label{eq:energy_packed}
E(\vec u(t)) - E(q,0) = 2\alpha \int_t^{+\infty} \| \partial_t u(s) \|_{L^2}^2 \d s.
\end{equation}
\end{proposition}

Notice that it is unclear whether a packed solution is globally bounded in $H^1 \times L^2$ (recall that from arguments of \cite{CaNLKG} -- see also \cite{BRS} and \cite{CMY} -- if $p \le \frac{N}{N-2}$, then any global solution to \eqref{equDKG} is globally bounded in $H^1 \times L^2$, but this is not known for higher powers of $p$).

\smallskip

It turns out that the description of the convergence depends deeply on the bound state. More specifically, consider the linearized operator ${\q L}_{q}$ of the energy around a bound state $q$:
\begin{equation}
{\mathcal{L}_{q}}=-\Delta+1-f'(q),\quad \langle {\mathcal{L}}_{q}v,v\rangle=\int_{\RR^{N}}\left\{|\nabla v|^{2}+v^{2}-f'(q)v^{2}\right\}\d x.
\end{equation}
Due to the invariances of equations, $\q L_{q}$ always has a important kernel: denote the $\Omega_{ij}$ are the angular derivatives that are
\begin{equation}
\Omega_{ij}=x_{i}\partial_{x_{j}}-x_{j}\partial_{x_{i}}\quad \text{for}\ 1\le i<j\le N,
\end{equation}
and consider the vector space ${\q Z}_q$ spanned by the infinitesimal generator of the invariance of the equation on $q$:
\begin{equation}
{\mathcal{Z}}_{q} ={\Span}\left\{\partial_{x_{n}}q,n=1,\cdots,N;\Omega_{ij}q,1\le i<j\le N\right\}.
\end{equation}
One always has ${\q Z}_{q}\subset\ker\q L_q$. Then we define non-degenerate and degenerate state.

\begin{definition}
Let $q\in \mathcal{B}$.
\begin{enumerate}
\item[$\rm{(i)}$] $q$ is called a \emph{non-degenerate state} if $\mathcal{Z}_{q}=\ker\mathcal{L}_{q}$.
\item [$\rm{(ii)}$] $q$ is called a \emph{degenerate state} if $\mathcal{Z}_{q}\subsetneq \ker\mathcal{L}_{q}$.
\end{enumerate}
\end{definition}

The most relevant example is of course the ground state $q_{0}$ which is non-degenerate. We will comment further on degenerate excited state in the comment paragraph below. For now, let us simply mention one way to understant degeneracy (we denote $'$ for Gateau differentials). The condition that $q$ is a bound state writes $E'(q)=0$. Then $\langle \q L_{q} v,w \rangle = E''(q)\cdot(v,w)$, so that the condition that $q$ is degenerate, is equivalent to the fact that for some $\phi \notin {\q Z}_{q}$, the linear form $E''(q)\cdot\left(\phi,\cdot\right) =0$.

\smallskip

Our first result is that if one cluster point of a packed solution is a non-degenerate state, then the convergence holds of all positive time, and occurs with an exponential rate. More precisely, we have the following.

\begin{theorem}\label{thm:1}
Let $\vec u=(u,\partial_{t}u)$ be a packed solution of~\eqref{equDKG}, with cluster point $q \in H^1$ at $(t_{n},y_{n})_{n}$.
If $q$ is a non-degenerate state, we have convergence holding for all time and exponential decay, i.e. there exist $\mu>0$ and $z_{\infty}\in \RR^{N}$ such that
	\begin{equation*}
	\forall t \ge 0, \quad \|u(t)-q(\cdot-z_{\infty})\|_{H^{1}}+\|\partial_{t}u(t)\|_{L^{2}}\lesssim e^{- \mu t}.
	\end{equation*}
\end{theorem}

Next, we consider \emph{degree-1 excited states} where $\ker \q L_q$ has one extra dimension not related to the geometric invariances of the equation \eqref{equDKG} and which also involves a condition on the third-order Gateau differentials of $E$, according to the next definition.

\begin{definition} \label{def:phi}
Let $q$ be a degenerate excited state. $q$ is called a \emph{degree-1 excited state} if there exists $\phi \in H^{1}$ such that 
\begin{equation} \label{def:deg1}
\ker \q L_{q}={\q Z}_{q}\oplus \Span \{\phi\} \quad \text{and} \quad E'''(q) \cdot(\phi,\phi,\phi) \ne 0.
\end{equation}
\end{definition}

Again, we will comment on this definition in the paragraph below, the main point being that degree-1 excited states are somehow the simplest degenerate bounds states.

Our second result is concerned with cluster points which are degree-1 excited states.

\begin{theorem}\label{thm:2}
Let $\vec u=(u,\partial_{t}u)$ be a packed solution of~\eqref{equDKG}, with cluster point $q \in H^1$ at $(t_{n},y_{n})_{n}$.
If $q$ is a degree-1 excited state, then the convergence $\vec u(t) \to (q,0)$ holds for all time as $t \to +\infty$, and the rate of convergence has algebraic decay, i.e. there exists $z_{\infty}\in \RR^{N}$ such that
\begin{equation*}
\forall t > 0, \quad \|u(t)-q(\cdot-z_{\infty})\|_{H^{1}}+\|\partial_{t}u(t)\|_{L^{2}}\lesssim t^{-1}.
\end{equation*}
\end{theorem}

Last, we show that the convergence rate in Theorem \ref{thm:2} can be sharp: we  provide an example where the solution converges exactly at the rate $t^{-1}$ to the degree-1 excited state.
\begin{theorem}\label{thm:3}
Let $q$ be a degree-1 excited state. Then, there exists a global solution $\vec{u}=(u,\partial_{t}u) \in \mathscr C([0,+\infty), H^1 \times L^2)$ of~\eqref{equDKG} such that 
	\begin{equation*}
\|u(t)-q\|_{H^{1}}+\|\partial_{t} u(t)\|_{L^{2}}\sim t^{-1} \quad \text{as} \quad t \to +\infty.
	\end{equation*}
\end{theorem}

\subsection{Comments}

Let us first observe that Theorem \ref{thm:1} (and its proof) holds also in dimension 1, but of course, they are in that case a direct consequence of the complete description \cite{CMY} of global solutions in 1D (as mentioned above, excited states only exist for $N \ge 2$). This is the only reason why we restrict to dimension $N \ge 2$. The restrictions to $N \le 5$ and $2 <p < \frac{N+2}{N-2}$ are to ensure a nice local well posedness theory, and sufficient smoothness on the non-linearity so that Taylor expansion make sense up to order 2. In this perspective let us remind that our analysis encompasses the most physically relevant nonlinearity, the cubic one $f(u) = u^3$.

\bigskip

Regarding Theorem \ref{thm:1}: the ground state is of course non-degenerate, but one should keep in mind that is not so easy to construct degenerate excited states. As a matter of fact, the constructions in \cite{AMW,MPW} (see also \cite{DMPP11,MW} for the massless case) yield \emph{non-degenerate} excited states as well. This means that the scope of Theorem \ref{thm:1} is rather large and does certainly not restrict to the ground state.

\bigskip

We now discuss degree-1 excited states: as we mentioned, they should be understood as the simplest degenerate case. Already here, very little is known, and to our knowledge, our results are the first describing precisely the dynamics in a degenerate setting. From this point of view, the condition that $\dim \ker \q L_q = \dim (\q Z_q) +1$  is very natural.  Regarding the extra condition $E{'''}(q)\cdot(\phi,\phi,\phi) \ne 0$, let us note that it is generic; as we will see in Lemma \ref{le:spec}, it is equivalent with $E{'''}(q)$ being non identically $0$ on $(\ker \q L_q)^3$. 

It is remarquable that one already observes a drastic change in the dynamic in degree-1 degeneracy, when compared to non-degeneracy. The convergence here is indeed merely polynomial in time, which is a surprise: such slow rate of convergence is usually observed due to the interaction with another nonlinear object (as in \cite{CMYZ,CMY}), and this is not the case here. As it is seen in the proofs, the derivation of the main bootstrap regime is noticeably more involved in degree-1 degeneracy, and relies on the very specific algebra of the main ODE system at leading order (see Section \ref{sec:3}). 

\bigskip

One setting where excited states are better understood is the case of radial functions. Among these, radial bound states $q$ are either non-degenerate or satisfy the first condition in the degree-1 degeneracy definition \eqref{def:deg1}: indeed, among radial functions, the geometric kernel $\q Z_q$ is trivial and $\dim \ker_{\text{rad}} \q L_q \le 1$, see for example \cite[Section 2.3]{BRS}. 

All the arguments in the proofs below can taken word for word to the radial setting, and so our results hold for any packed \emph{radial} solution converging to a radial bound state $q$ which is either non-degenerate (Theorem \ref{thm:1}) or such that $E{'''}(q)|_{(\ker \q L_q)^3} \ne 0$ (Theorems \ref{thm:2} and \ref{thm:3}).

\section{Preliminaries}

\subsection{Proof of Proposition~\ref{prop:1.1}}\label{se:prop}

\begin{proof}[Proof of Proposition~\ref{prop:1.1}]
Denote $\vec W(t)=(W(t),\partial_{t} W(t))$ the solution to \eqref{equDKG} with initial data $\vec W(0) = (W_0,W_1)$ and $\vec u_n(t,x) = \vec u(t_n+t,x+y_n)$ the solution to \eqref{equDKG} with initial data $\vec u_n(0,x) = \vec u(t_n,x+y_n)$. We can assume that $\vec W \in \mathscr C([0,T_0], H^1 \times L^2)$ for some $T_0>0$.

As the \eqref{equDKG} flow is continuous in $H^1 \times L^2$ and $\vec u_n(0) \to \vec W(0)$, we infer that $\vec u_n$ is defined on $[0,T_0]$ for $n$ large enough and that \begin{equation} \label{conv_un}
 \vec u_n \to \vec W \quad \text{in} \quad \mathscr C([0,T_0], H^1 \times L^2).
 \end{equation}

This immediately prove that $\vec u$ is globally defined for positive times. Indeed, if $T_{\rm{max}} <+\infty$, then for large enough $n$, $T_{\rm{max}} \ge t_n + T_0 \to T_{\rm{max}} + T_0$, a contradiction: hence $T_{\rm{max}} = +\infty$.

As $\vec E(\vec u (t_n)) \to \vec E(W_0,W_1)$ we infer from the energy dissipation identity that $\partial_t u \in L^2([0,+\infty),L^2)$. Assume that $\vec W$ is not a stationnary solution. Then we can furthermore assume that $\partial_t W \ne 0$ on $[0,T_0] \times \m R^N$, so that $\| \partial_t W \|_{L^2([0,T_0]),L^2)} >0$. In particular, from the convergence \eqref{conv_un}, we conclude that
\[ \| \partial_t u \|_{L^2([t_n,t_n+T_0],L^2)} = \| \partial_t u_n \|_{L^2([0,T_0]),L^2)} \to  \| \partial_t W \|_{L^2([0,T_0]),L^2)} \quad \text{as} \quad n \to +\infty. \]
Let $t_n'$ be a subsequence of $t_n$ such that for all $n \in \m N$, $t_{n+1}' \ge t_n' +T_0$ and 
\[ \| \partial_t u \|_{L^2([t_n',t_n'+T_0],L^2)} \ge \frac{1}{2}  \| \partial_t W \|_{L^2([0,T_0]),L^2)}. \]
There holds
\[ \| \partial_t u \|_{L^2([0,+\infty),L^2)}^2 \ge \sum_{n} \| \partial_t u \|_{L^2([t_n',t_n'+T_0],L^2)}^2 \ge \sum_n \frac{1}{4} \| \partial_t W \|_{L^2([0,T_0]),L^2)}^2 = +\infty, \]
which is a contradiction. As a consequence, $\vec W$ is a stationary solution, which means that for all $t \ge0$, $\partial_t W(t)=0$ and $W(t)=q$ for some bound state $q$. In particular, $W_0 =q$ and $W_1=0$.

Furthermore, the energy dissipation identity writes for all $0 \le t \le t_n$:
\[ E(\vec u(t)) - E(\vec u(t_n)) = 2 \alpha \int_t^{t_n} \| \partial_t u (s) \|_{L^2}^2 \d s. \]
Letting $n \to +\infty$, we see that the left-hand side has a limit $E(\vec u(t)) - E(q,0)$, and so $\partial_t u \in L^2([t,+\infty),L^2)$. This completes the proof of Proposition \ref{prop:1.1}.
\end{proof}

\subsection{Notation}
Let $q$ be a bound state. Let $I_q$ be a subset of $\{ (i,j) : 1 \le i < j \le N \}$ such that
\[ \left\{\partial_{x_{n}}q,1\le n\le N;\Omega_{ij} q,(i,j)\in I_{q}\right\} \text{ is a basis of } \mathcal{Z}_{q}. \]
For any $(i,j)\in I_{q}$ and $\vartheta \in\m R$, we recall the Givens rotation:
\begin{equation}\label{def:Gij}
G_{ij}(\vartheta)=\left(\begin{matrix}
1      & \cdots & 0      & \cdots & 0     & \cdots & 0\\
\vdots & \ddots & \vdots &        &\vdots &        & \vdots\\
0      & \cdots & \cos\vartheta      &\cdots  & - \sin \vartheta    & \cdots & 0\\
\vdots &        & \vdots & \ddots & \vdots&        &\vdots\\
0      & \cdots & \sin \vartheta      &\cdots  & \cos \vartheta     & \cdots & 0\\
\vdots &        & \vdots &        &\vdots & \ddots &\vdots\\
0      & \cdots & 0      & \cdots & 0     & \cdots & 1
\end{matrix}\right),
\end{equation}
where $\cos \vartheta$ and $\sin\vartheta$ appear at the intersections $i$th and $j$th rows and columns. That is, the non-zero elements of the Givens matrix $G_{i,j}(\vartheta) = (g_{nm})_{nm}$ are given by:
\begin{equation*}
g_{nn}=1\quad \text{for } n\ne i,j,\quad 
g_{ii}=g_{jj}=\cos\vartheta,\quad \text{and}\quad 
g_{ij}=-g_{ji}=-\sin \vartheta.
\end{equation*}
 
For $K\in \mathbb{N}^*$ and $r>0$, we denote by $\mathcal{B}_{\RR^{K}}(r)$ (respectively, $\mathcal{S}_{\RR^{K}}(r)$) be the ball (respectively, the sphere) of $\RR^{K}$ of center $0$ and of radius $r$.

We denote $\langle\cdot,\cdot \rangle$ the $L^{2}$ scalar product for real-valued functions $u,v\in L^{2}$,
\begin{equation*}
\langle u,v \rangle:=\int_{\RR^{N}}u(x)v(x)\d x.
\end{equation*}
For vector-valued functions
\begin{equation*}
\vec{u}=\left( \begin{array}{c}
u_{1}\\ u_{2}
\end{array}\right),\quad \vec{v}=\left( \begin{array}{c}
v_{1}\\ v_{2}
\end{array}\right),
\end{equation*}
the notation $\langle\cdot,\cdot \rangle$ is also the $L^2$ scalar product, 
\begin{equation*}
\langle \vec{u},{\vec{v}}\rangle:=\sum_{k=1,2}\langle u_{k},v_{k}\rangle,\quad 
\|\vec{u}\|^{2}_{\E}:=\|u_{1}\|_{H^{1}}^{2}+\|u_{2}\|_{L^{2}}^{2}.
\end{equation*}
We also define $\bar{p}=\min\{3,p\}>2$ (recall that the nonlinearity power $p >2$).

\subsection{Spectral theory of linearized operator}

In this section, we introduce some spectral properties of the linearized operator for any bound state $q\in \mathcal{B}$.

For $\theta=\left(\theta_{ij}\right)_{(i,j)\in I_{q}}\in \mathbb{R}^{\# I_{q}}$, denote the rotation
\begin{equation*}
	R_{\theta}=G_{i_{1}j_{1}}\left(\theta_{i_{1}j_{1}}\right)\cdots G_{i_{\# I_{q}}j_{\# I_q}} (\theta_{ 
		i_{\# I_{q}} j_{\# I_{q}}}).
\end{equation*}
For $(z,\theta)\in \mathbb{R}^{N + \#I_q}$, we introduce the following transformation $\mathcal{T}_{(z,\theta)}$ linked to the symmetries of \eqref{equDKG}: for $f \in L^2$,
\begin{equation*}
\mathcal{T}_{(z,\theta)}f:=f(R_{\theta}(\cdot-z)).
\end{equation*}
Observe that for all $q\in \mathcal{B}$, ${\mathcal{Z}}_{q}$ is generated by taking partial derivatives of $\mathcal{T}_{(z,\theta)}  q$ with respect to $(z,\theta)$ at $(z,\theta)=(\boldsymbol{0},\boldsymbol{0})$:
\begin{align} \label{eq:derivative_T}
\partial_{x_n} q = - \frac{\partial}{\partial z_n} \mathcal{T}_{(z,\theta)} q |_{(z,\theta) = (\boldsymbol{0},\boldsymbol{0})}, \quad \Omega_{ij} q = \frac{\partial}{\partial \theta_{ij}} \mathcal{T}_{(z,\theta)} q |_{(z,\theta) = (\boldsymbol{0},\boldsymbol{0})}.
\end{align}

First, we recall standard properties of the linearized operator $\mathcal{L}_{q}$.
\begin{lemma}\label{le:spec}
	\emph{(i) Spectral properties.} The self-adjoint operator $\mathcal{L}_{q}$ has essential spectrum $[1,+\infty)$, a finite number $K\ge 1$ of negative eigenvalues and its kernel is of finite dimension $M$ with $M\ge N$. Let $(Y_{k})_{k=1,\cdots,K}$ be an $L^{2}$ orthogonal family of eigenfunctions of $\mathcal{L}_{q}$ with negative eigenvalues $(-\lambda^{2}_{k})_{k=1,\cdots,K}$, i.e.
	\begin{equation}\label{equ:Yk}
	\langle Y_{k},Y_{k'}\rangle=\delta_{kk'}\quad \text{and}\quad \mathcal{L}_{q}Y_{k}=-\lambda_{k}^{2}Y_{k},\quad \lambda_{k}>0.
	\end{equation}
\emph{(ii) Coercivity.} Denote $\Pi_q$ the $L^2$-orthogonal projection on $\ker \q L_q$. There exists $c>0$ such that for all $\eta\in H^{1}$,
\begin{equation}\label{coer}
\langle \mathcal{L}_{q}\eta,\eta\rangle\ge c \|\eta\|^{2}_{H^{1}}-c^{-1}\bigg(\| \Pi_q \eta \|_{L^2}^{2}+\sum_{k=1}^{K}\langle \eta,Y_{k}\rangle^{2}\bigg).
\end{equation}
\emph{(iii) Cancellation.}
We have, for all $\psi_{1},\psi_{2}\in {\ker\mathcal{L}_{q}}$ and $\psi_{3}\in {\mathcal{Z}}_{q}$
\begin{equation}\label{equ:can}
\langle f''(q)\psi_{1}\psi_{2},\psi_{3}\rangle=0.
\end{equation}
\end{lemma}
\begin{proof}
	Proof of (i) and (ii). See the proof of~\cite[Lemma 1]{CMakg}.
	
	Proof of (iii). Without loss of generality, we first consider
	\begin{equation*}
	\psi_{3}=\Omega_{ij}q\quad \text{for}\ (i,j)\in I_{q}.
	\end{equation*}
For any $\psi_{1}\in \ker\mathcal{L}_{q}$, we have
\begin{equation*}
-\Delta \psi_{1}+\psi_{1}-f'(q)\psi_{1}=0.
\end{equation*}
Consider the transformation $\mathcal{T}_{(z,\theta)}$ with $(z,\theta)=\left(\boldsymbol{0},\theta\right)$ for the above identity,
\begin{equation*}
-\Delta (\mathcal{T}_{(z,\theta)}\psi_{1})+\mathcal{T}_{(z,\theta)}\psi_{1}-f'(\mathcal{T}_{(z,\theta)}q)(\mathcal{T}_{(z,\theta)}\psi_{1})=0.
\end{equation*}
Note that, from $p>2$, we can take the derivative of above identity with respect to $\theta_{ij}$, and then let $\theta=\boldsymbol{0}$. It follows that 
\begin{equation*}
f''(q)\psi_{1}\psi_{3}=-\Delta \tilde{\psi}_{1}+\tilde{\psi}_{1}-f'(q)\tilde{\psi}_{1}=\mathcal{L}_{q}\tilde{\psi}_{1}\quad \text{where}\ \tilde{\psi}_{1}=\Omega_{ij}\psi_{1}.
\end{equation*}
Thus, by integration by parts and $\psi_{2}\in {\ker\mathcal{L}_{q}}$,
\begin{equation*}
\langle f''(q)\psi_{1}\psi_{2},\psi_{3}\rangle=\langle f''(q)\psi_{1}\psi_{3},\psi_{2}\rangle=\langle \mathcal{L}_{q}\tilde{\psi}_{1},\psi_{2}\rangle=\langle \tilde{\psi}_{1},\mathcal{L}_{q}\psi_{2}\rangle=0.
\end{equation*}
Proceeding similarly for all the parameters in the transformation $\mathcal{T}_{(z,\theta)}$, we complete the proof of (iii).
\end{proof}

As a consequence of the above Lemma,  in the case when $q$ is furthermore assumed to be a degree-1 excited state, we now choose $\phi$ (introduced in Definition \ref{def:phi}) with more rigid properties: namely we claim that there exists (a unique)  $\phi \in H^1$ such that for all $n=1,\cdots,N$, $(i,j)\in I_{q}$,
\begin{equation}\label{equ:phi}
\langle \phi,\partial_{x_{n}}q\rangle=\langle \phi,\Omega_{ij}q\rangle=0\quad 
\text{and}\quad 
\ker \q L_{q}={\q Z}_{q} {\oplus} \Span\{\phi\},
\end{equation}
and such that 
\begin{equation}\label{def:kappa}
-(4\alpha\|\phi\|_{L^{2}}^{2})^{-1}E'''(\phi,\phi,\phi) =1.
\end{equation}
\eqref{equ:phi} essentially means that $\phi$ is the $L^2$-orthogonal supplement of $\q Z_q$. Moreover, due to \eqref{equ:can} and the fact that $E'''(q)$ is not identically $0$ on $\q L_q^3$ (in view of \eqref{def:deg1}), for such a $\phi$ one has $E'''(q)(\phi,\phi,\phi) \ne 0$: and so, by considering the transformations $\phi\to -\phi$ and $\phi\to \lambda \phi$, the  condition \eqref{def:kappa} can be met.

 \smallskip

\subsection{Modulation around a bound state}

Let $q \in \q B$ be a degree-1 degenerate bound state. Given  time dependent $\mathscr C^1$ functions $z,\theta,a$, with values in $\mathbb{R}^{N}$, $\mathbb{R}^{\# I_{q}}$ and $\m R$, we denote
\begin{align}
Q=\mathcal{T}_{(z,\theta)}q &=q\left(R_{\theta}(\cdot-z)\right), & \quad & \Phi(t,x) =  \q T_{(z,\theta)} \phi = \phi \left(R_{\theta}(\cdot-z)\right),\\
V(t,x) & = a(t) \Phi(t,x), & \quad & G=f\left(Q+V\right)-f(Q)-f'(Q) V. \label{def:G}
\end{align}
It will convenient to encompass both non-degenerate and degree-1 degenerate cases at once by setting $a \equiv 0$ if $q$ is non-degenerate.

For all $(i,j)\in I_{q}$ and $(i',j')\in I_{q}$, we denote as follows the derivatives:
\begin{align}
\Psi_{ij}=\frac{\partial Q}{\partial{\theta_{ij}}},\quad \Phi_{ij}=\frac{\partial \Phi}{\partial{\theta_{ij}}},\quad 
\Psi_{ij}^{i'j'}=\frac{\partial \Psi_{ij}}{\partial \theta_{i'j'}}.
\end{align}
Finally, we introduce the exponential directions. For $k=1,\cdots,K$, we denote
\[ \Upsilon_{k} =\mathcal{T}_{(z,\theta)}Y_{k},\quad \Upsilon_{k}^{ij}=\frac{\partial \Upsilon_{k}}{\partial \theta_{ij}}, \]
and
\begin{equation}
\nu_{k}^{\pm}=-\alpha\pm\sqrt{\alpha^{2}+\lambda_{k}^{2}},\quad \zeta_{k}^{\pm}=\alpha\pm \sqrt{\alpha^{2}+\lambda^{2}_{k}}\quad \text{and}\quad 
\vec{Z}^{\pm}_{k}= \begin{pmatrix}
\zeta_{k}^{\pm}\Upsilon_{k}\\
\Upsilon_{k}
\end{pmatrix}.
\end{equation}
The importance of $\vec{Z}^\pm_k$ come from the following observation: if $\vec v = (v_1,v_2)$ is a solution to the linearized \eqref{equDKG} equation
\[ \partial_t \begin{pmatrix} 
v_1 \\
v_2
\end{pmatrix}  = \begin{pmatrix} 
v_2 \\
- 2 \alpha v_2 - \q L_q v_1
\end{pmatrix}, \]
then with $a_k^\pm := \langle \vec v, \vec Z_k^\pm \rangle$, there hold $\frac{\d}{\d t} a_k^\pm = \nu_k^\pm a_k^\pm$. 
\smallskip

Observe that all the function introduced are at least of class $\mathscr C^1$ and have pointwise exponential decay.

By direct computation and the definition of Givens rotations in~\eqref{def:Gij}, we have, for all $n=1,\cdots,N$ and $(i,j)\in I_{q}$, 
\begin{equation}\label{equ:kert}
\begin{aligned}
\Psi_{ij}&\in {\Span} \left\{(\Omega_{ij}q)\left(R_{\theta}(\cdot-z)\right): (i,j)\in I_{q}\right\},\\
\partial_{x_{n}}Q&\in {\Span}
\left\{(\partial_{x_{1}}q)\left(R_{\theta}(\cdot-z)\right),\cdots,(\partial_{x_{N}}q)\left(R_{\theta}(\cdot-z)\right)\right\}.
\end{aligned}
\end{equation}
Moreover, by the chain rule, we have, for all $(i,j)\in I_{q}$ and $k=1,\cdots,K$,
\begin{equation}\label{equ:dtQPhiY}
\begin{aligned}
\partial_{t} Q&=-\dot z\cdot \nabla Q+\sum_{(i,j)\in I_{q}}\dot{\theta}_{ij}\Psi_{ij},\\
\partial_{t}\Phi&=-\dot z\cdot\nabla \Phi+\sum_{(i,j)\in I_{q}}\dot{\theta}_{ij}\Phi_{ij},\\
\partial_{t}\Upsilon_{k}&=-\dot{z}\cdot\nabla \Upsilon_{k}+\sum_{(i,j)\in I_{q}}\dot{\theta}_{ij}
\Upsilon^{ij}_{k},\\
\partial_{t}\Psi_{ij}&=-\dot{z} \cdot \nabla \Psi_{ij}+\sum_{(i',j')\in I_{q}}\dot{\theta}_{i'j'}\Psi_{ij}^{i'j'}.
\end{aligned}
\end{equation} 

As a consequence of \eqref{eq:derivative_T}, we have the following expansions for small $\theta \in \m R^{\# I_q}$.
\begin{lemma}
For $|\theta|\ll 1$ small, we have, for all $(i,j)\in I_{q}$, $(i',j')\in I_{q}$, $k=1,\dots,K$ and $n=1,\dots,N$,
\begin{equation}\label{est:kernel}
\begin{aligned}
\Psi_{ij}&=\left(\Omega_{ij}q\right)\left(R_{\theta}(\cdot-z)\right)+O_{H^{1}}(|\theta|),\\
\Phi_{ij}&=\left(\Omega_{ij}\phi\right) \left(R_{\theta}(\cdot-z)\right)+O_{H^{1}}(|\theta|),\\
 \Upsilon_{k}^{ij}&=\left(\Omega_{ij}Y_{k}\right)\left(R_{\theta}(\cdot-z)\right)+O_{H^{1}}(|\theta|),\\
\partial_{x_{n}}Q&=\left(\partial_{x_{n}}q\right)\left(R_{\theta}(\cdot-z)\right)+O_{H^{1}}(|\theta|),\\
\Psi_{ij}^{i'j'}&=\left(\Omega_{i'j'}\Omega_{ij}q\right)\left(R_{\theta}(\cdot-z)\right)+O_{H^{1}}(|\theta|).
\end{aligned}
\end{equation}
\end{lemma}

If $q$ is a non-degenerate bound state, we use same notations as above for degree-1 degenerate bound states, but with $a=0$ and $\phi=0$.

For future reference, we state the following Taylor formulas involving the functions $F$ and $f$, and omit its proof.

\begin{lemma} For all $s\in \mathbb{R}$, and $x \in \m R^N$, we have 
\begin{gather}\label{est:tay1}
\left|f'(Q+s)-f'(Q)\right|\lesssim |s|+|s|^{p-1}, \\
\label{est:tay2}
\left|f(Q+s)-f(Q)-f'(Q)s\right|\lesssim s^{2}+|s|^{p}, \\
\label{est:tay6}
\left|f(Q+s)-f(Q)-f'(Q)s-\frac{1}{2}f''(Q)s^{2}\right|\lesssim |s|^{3}+|s|^{p}, \\
\label{est:tay3}
\left|f(Q+V+s)-f(Q+V)-f'(Q+V)s\right|\lesssim s^{2}+|s|^{p}, \\
\label{est:tay4}
\left|F(Q+V+s)-F(Q+V)-f(Q+V)s\right|\lesssim |Q+V|^{p-1}s^{2}+|s|^{p+1}, \\
\label{est:tay5}
\left|F(Q+V+s)-F(Q+V)-f(Q+V)s-\frac{1}{2}f'(Q+V)s^{2}\right|\lesssim |s|^{3}+|s|^{p+1},
\end{gather}
where all the implied constants in the $\lesssim$ are uniform in the space variable of $Q$ or $V$.
\end{lemma}

First, we introduce the standard modulation result around the non-degenerate state or degree-1 excited state $q$.

\begin{proposition}[Properties of the modulation]\label{pro:dec}
		There exists $0<\gamma_{0}\ll 1$ such that for any $0<\gamma<\gamma_{0}$, $T_{1}\le T_{2}$, and any 
		solution $\vec{u}=(u,\partial_{t} u)$ of~\eqref{equDKG} on $[T_{1},T_{2}]$ satisfying
	\begin{equation}\label{est:dec}
		\sup_{t\in[T_{1},T_{2}]}\left\{\inf_{\xi\in \RR^{N}}\|u(t)-q(\cdot-\xi)\|_{H^{1}}+\|\partial_{t}u(t)\|_{L^{2}}\right\}<\gamma,
	\end{equation}
		there exist unique $\mathscr C^{1}$ functions 
		\begin{equation*}
		\begin{array}{rcl}
		[T_{1},T_{2}] & \to & \mathbb{R}^{N}\times\mathbb{R}^{N}\times \mathbb{R}^{\# I_{q}}\times \mathbb{R}^{\# I_{q}}\times \mathbb{R}\times\mathbb{R} \\
		t & \mapsto & (z(t),\ell(t),\theta(t),\beta(t),a(t),b(t)) 
		\end{array}
		\end{equation*}
		such that, if we define $\vec{\varphi}=(\vp,\vpp)$ by 
		\begin{equation}\label{def:vp}
		\vec{u}=\left(
		\begin{array}{c}
		u\\
		\partial_{t} u
		\end{array}\right)=\left(
		\begin{array}{c}
		Q\\
		-\ell\cdot\nabla Q
		\end{array}\right)+\sum_{(i,j)\in I_{q}}\beta_{ij}\left(
		\begin{array}{c}
		0\\
		\Psi_{ij}
		\end{array}\right)+\left(
		\begin{array}{c}
		a\Phi\\
		b\Phi
		\end{array}\right)+\left(
		\begin{array}{c}
		\varphi_{1}\\
		\varphi_{2}
		\end{array}\right),
		\end{equation}
where $\beta = (\beta_{ij})_{(i,j) \in I_q}$, it satisfies, for all $t\in [T_{1},T_{2}]$,
		\begin{equation}\label{small}
		\begin{aligned}
			\|\vec{\varphi}(t)\|_{\E}+|\theta(t)|&\lesssim \|u(t)-q\|_{H^{1}}+\|\partial_{t}u(t)\|_{L^{2}}\lesssim \gamma,\\
			|\ell(t)|+|\beta(t)|+|a(t)|+|b(t)|&\lesssim \|u(t)-q\|_{H^{1}}+\|\partial_{t}u(t)\|_{L^{2}}\lesssim \gamma,
			\end{aligned}
		\end{equation} 
	and for all $n=1,\cdots,N$ and $(i,j)\in I_{q}$,
	\begin{equation}\label{equ:orthvp}
	\langle \varphi_{1},\partial_{x_{n}}Q\rangle=\langle \varphi_{1},\Psi_{ij}\rangle=\langle \vp,\Phi\rangle=0,
	\end{equation}
	\begin{equation}\label{equ:orthvpp}
		\langle \varphi_{2},\partial_{x_{n}}Q\rangle=\langle \varphi_{2},\Psi_{ij}\rangle=\langle \vpp,\Phi\rangle=0.
	\end{equation}
\end{proposition}
\begin{proof}
   The proof of the decomposition result relies on a standard argument based
	on the Implicit function Theorem (See \emph{e.g.}~\cite[Appendix B]{CMkg}) and we omit it.
	\end{proof}

Second, we derive the equation of $\vec{\varphi}$ from~\eqref{equDKG} and~\eqref{def:vp}.

\begin{lemma}[Equation of $\vec{\varphi}$]
	In the contex of Proposition \ref{pro:dec}, we have
	\begin{equation}\label{equ:vp}
	\left\{ \begin{aligned}
	&\partial_{t}\varphi_{1}=\varphi_{2}+\Mod_{1}+G_{1},\\
	&\partial_{t}\varphi_{2}=\Delta \varphi_{1}-\varphi_{1}-2\alpha\varphi_{2}+f(Q+V+\vp)-f(Q+V)+\Mod_{2}+G_{2}+G,
	\end{aligned}\right.
	\end{equation}
	where
	\begin{align*}
	\Mod_{1}
	&:=\big(\dot{z}-\ell\big) \cdot \nabla Q-\sum_{(i,j)\in I_{q}}\big(\dot{\theta}_{ij}-\beta_{ij}\big)\Psi_{ij}-(\dot{a}-b)\Phi,\\
	\Mod_{2} & :=\big(\dot{\ell}+2\alpha\ell\big)\cdot \nabla Q-\sum_{(i,j)\in I_{q}}\big(\dot{\beta}_{ij}+2\alpha \beta_{ij}\big)\Psi_{ij}-\left(\dot{b}+2\alpha b\right)\Phi,
	\end{align*}
	and $G = f(Q+V) - f(Q) - f'(Q)V$ is defined in \eqref{def:G}, 
	\begin{equation*}
	\begin{aligned}
	G_{1}:=&a\dot{z}\cdot\nabla \Phi-a\sum_{(i,j)\in I_{q}}\dot{\theta}_{ij}\Phi_{ij},\\
	G_{2} :=&\sum_{(i,j)\in I_{q}}\dot{\theta}_{ij}\left(\ell\cdot\nabla \Psi_{ij}\right)+\sum_{(i,j)\in I_{q}}\beta_{ij}\left(\dot{z}\cdot\nabla \Psi_{ij}\right)-\left(\ell\cdot\nabla\right)\left(\dot{z}\cdot\nabla \right)Q\\
	&-\sum_{(i',j')\in I_{q}}\sum_{(i,j)\in I_{q}}\dot{\theta}_{i'j'}{\beta}_{ij}\Psi_{ij}^{i'j'}+b\dot{z}\cdot\nabla \Phi-\sum_{(i,j)\in I_{q}}b\dot{\theta}_{ij}\Phi_{ij}.
	\end{aligned}
	\end{equation*}
\end{lemma}
\begin{proof}
	First, from~\eqref{equ:dtQPhiY} and~\eqref{def:vp},
	\begin{equation*}
	\begin{aligned}
	\partial_{t}\varphi_{1}
	=&\partial_{t}u-\partial_{t}Q-a\partial_{t}\Phi-\dot{a}\Phi\\
	=&\vpp+\big(\dot{z}-\ell\big) \cdot \nabla Q-\sum_{(i,j)\in I_{q}}\big(\dot{\theta}_{ij}-\beta_{ij}\big)\Psi_{ij}-(\dot{a}-b)\Phi\\
	&+a\dot{z}\cdot\nabla \Phi-a\sum_{(i,j)\in I_{q}}\dot{\theta}_{ij}\Phi_{ij} =
	\vpp+\Mod_{1}+G_{1}.
	\end{aligned}
	\end{equation*}
	Using~\eqref{equ:dtQPhiY} and~\eqref{def:vp} again,
	\begin{equation*}
	\begin{aligned}
	\partial_{t}\vpp=&\partial_{tt}u+\dot{\ell}\cdot \nabla Q-\sum_{(i,j)\in I_{q}}\dot{\beta}_{ij}\Psi_{ij}-\dot{b}\Phi\\
	&+\ell\cdot \nabla \partial_{t}Q-\sum_{(i,j)\in I_{q}}\beta_{ij}\partial_{t}\Psi_{ij}-b\partial_{t}\Phi\\
	=&\partial_{tt}u+\dot{\ell}\cdot \nabla Q-\sum_{(i,j)\in I_{q}}\dot{\beta}_{ij}\Psi_{ij}-\dot{b}\Phi+G_{2}.
	\end{aligned}
	\end{equation*}
	From~\eqref{equDKG},~\eqref{def:vp},~$-\Delta Q+Q-f(Q)=0$ and~$-\Delta\Phi+\Phi-f'(Q)\Phi=0$,
	\begin{equation*}
	\begin{aligned}
	\partial_{tt}u=&\Delta u-u-2\alpha \partial_{t}u+f(u)\\
	=&\Delta \vp-\vp-2\alpha\vpp+f(Q+V+\vp)-f(Q+V)\\
	&+2\alpha \ell \cdot \nabla Q-2\alpha \sum_{(i,j)\in I_{q}} \beta_{ij}\Psi_{ij}-2\alpha b\Phi+G.
	\end{aligned}
	\end{equation*}
	Therefore, 
	\begin{align*}
	\partial_{t}\vpp & =\Delta \vp-\vp-2\alpha\vpp+f(Q+V+\vp)-f(Q+V)+G_{2}+G\\
	& \quad +\left(\dot{\ell}+2\alpha \ell\right)\cdot\nabla Q-\sum_{(i,j)\in I_{q}}\left(\dot{\beta}_{ij}+2\alpha \beta_{ij}\right)\Psi_{ij}-\left(\dot{b}+2\alpha b \right)\Phi. \qedhere
	\end{align*}
	\end{proof}

Third, we derive the control of geometric parameters from orthogonality conditions~\eqref{equ:orthvp} and~\eqref{equ:orthvpp}. Our goal here is to get an ODE system on the modulations parameters, at leading order, with bounds on the remainder terms as squares or higher powers of $|a|$ and
\begin{equation} \label{def:N}
\mathcal{N}: =\|\vec{\varphi}\|_{\E}+|\ell|+|\beta|+|b|.
\end{equation}
Recall that we defined $\bar{p}=\min\{3,p\}>2$.

\begin{lemma}\label{le:equzab}
	In the context of Proposition~\ref{pro:dec}, the following holds.
	\begin{enumerate}
	\item \emph{Control of non-degenerate directions.} We have 
	\begin{align}
	|\dot{z}-\ell|+\big|\dot{\theta}-\beta\big|
	&\lesssim  \mathcal{N}^{2}+|a|\mathcal{N},\label{equ:zta}\\
	|\dot{\ell}+2\alpha\ell|+|\dot{\beta}+2\alpha \beta|
	&\lesssim \mathcal{N}^{2}+|a|\mathcal{N}+|a|^{\bar{p}}.\label{equ:bl}
	\end{align}
	\item \emph{Control of extra direction.} For $q$ be a degree-1 excited state, we have
	\begin{align}
	|\dot{a}-b|&\lesssim \mathcal{N}^{2}+|a|\mathcal{N},\label{est:a}\\
	|\dot{b}+2\alpha b+2\alpha a^{2}|&\lesssim \mathcal{N}^{2}+|a|\mathcal{N}+|a|^{\bar{p}}.\label{est:ba2}
	\end{align}
\end{enumerate}
\end{lemma}

\begin{proof}
	Proof of (i). First, we differentiate the orthogonality $\langle \varphi_{1},\partial_{x_{n}}Q\rangle=0$ in~\eqref{equ:orthvp},
	\begin{equation*}
	\begin{aligned}
	0=&\frac{\d}{\d t}\langle\vp,\partial_{x_{n}}Q\rangle=\langle \partial_{t} \vp,\partial_{x_{n}}Q\rangle+\langle \vp,\partial_{t}\partial_{x_{n}}Q\rangle.
	\end{aligned}
	\end{equation*}
	Using~\eqref{equ:orthvp} and~\eqref{equ:vp},
	\begin{equation*}
	\begin{aligned}
	\langle \partial_{t} \vp,\partial_{x_{n}}Q\rangle 
	=\langle \Mod_{1},\partial_{x_{n}}Q\rangle
	+\langle G_{1},\partial_{x_{n}}Q\rangle.
	\end{aligned}
	\end{equation*}
	From~\eqref{equ:phi},~\eqref{equ:kert},~\eqref{est:kernel},~\eqref{small}, the expression of $\rm{Mod}_{1}$ and change of variables,
	\begin{equation*}
	\begin{aligned}
	&\langle \Mod_{1},\partial_{x_{n}}Q\rangle\\
	&= \langle (\dot{z}-\ell)\cdot\nabla Q,\partial_{x_{n}}Q\rangle
	-\sum_{(i,j)\in I_{q}}\left(\dot{\theta}_{ij}-\beta_{ij}\right)\langle \Psi_{ij},\partial_{x_{n}}Q\rangle-(\dot{a}-b)\langle \Phi,\partial_{x_{n}}Q\rangle\\
	&=\langle (\dot{z}-\ell)\cdot \nabla q, {\partial_{x_n} q} \rangle
	-\sum_{(i,j)\in I_{q}}\left(\dot{\theta}_{ij}-\beta_{ij}\right)\langle \Omega_{ij}q,\partial_{x_{n}}q\rangle
	+O(\gamma(|\dot{z}-\ell|+|\dot{\theta}-\beta|)).
	\end{aligned}
	\end{equation*}
	From the expression of $G_{1}$ and~\eqref{small},
	\begin{equation*}
	\begin{aligned}
	\left|\langle G_{1},\partial_{x_{n}}Q\rangle\right|
	&\lesssim |a|\left(|\dot{z}-\ell|+|\ell|+|\dot{\theta}-\beta|+|\beta|\right)\\
	&\lesssim \gamma \left(|\dot{z}-\ell|+|\dot{\theta}-\beta|\right)+|a|\mathcal{N}.
	\end{aligned}
	\end{equation*}
	Next, using again~\eqref{equ:dtQPhiY},
	\begin{equation}\label{equ:dtxq}
	\begin{aligned}
	\partial_{t}\partial_{x_{n}}Q&=\partial_{x_{n}}\partial_{t}Q=-\dot{z}\cdot\nabla\partial_{x_{n}}Q+\sum_{(i,j)\in I_{q}}\dot{\theta}_{ij}\partial_{x_{n}}\Psi_{ij}.
	\end{aligned}
	\end{equation}
	Thus, from~\eqref{small} and the Sobolev embedding Theorem,
	\begin{equation*}
	\begin{aligned}
	\left|\langle \vp,\partial_{t}\partial_{x_{n}}Q\rangle\right|
	&\lesssim \|\vec{\varphi}\|_{\E}\left(|\dot{z}-\ell|+|\ell|+|\dot{\theta}-\beta|+|\beta|\right)\\
	&\lesssim \gamma\left(|\dot{z}-\ell|+|\dot{\theta}-\beta|\right)+\mathcal{N}^{2}.
	\end{aligned}
	\end{equation*}
	Combining above estimates, we have
	\begin{equation*}
	\begin{aligned}
	&\langle (\dot{z}-\ell)\cdot \nabla q, \partial_{x_n} q\rangle
	-\sum_{(i,j)\in I_{q}}\left(\dot{\theta}_{ij}-\beta_{ij}\right)\langle \Omega_{ij}q,\partial_{x_{n}}q\rangle\\
	&=O\left(\gamma\left(|\dot{z}-\ell|+|\dot{\theta}-\beta|\right)+\mathcal{N}^{2}+|a|\mathcal{N}\right).
	\end{aligned}
	\end{equation*}
	{ This gives $N$ inequalities. Proceeding similarly for the $\# I_q$ orthogonality conditions $\langle \varphi_1, \Omega_{ij} Q \rangle =0$ in~\eqref{equ:orthvp}, and using the fact that the family 
	\[ \left\{\partial_{x_{n}}q,1\le n\le N; \Omega_{ij}q,(i,j)\in I_{q}\right\} \]
	 is linearly independent, so that its Gram matrix is invertible,} we obtain
	\begin{equation*}
	\left|\dot{z}-\ell\right|+\big|\dot{\theta}-\beta \big|\lesssim
	\gamma \big(\left|\dot{z}-\ell\right|+\big|\dot{\theta}-\beta\big|\big)
	+\mathcal{N}^{2}+|a|\mathcal{N},
	\end{equation*}
	which implies~\eqref{equ:zta}, upon taking $\gamma$ small enough.
	
	Second, we differentiate the orthogonality $\langle \vpp,\partial_{x_{n}}Q \rangle=0$ in~\eqref{equ:orthvpp},
	\begin{equation*}
	\begin{aligned}
	0=&\frac{\d}{\d t}\langle\vpp,\partial_{x_{n}} Q \rangle=\langle \partial_{t} \vpp,\partial_{x_{n}}Q\rangle
	+\langle \vpp,\partial_{t}\partial_{x_{n}}Q\rangle.
	\end{aligned}
	\end{equation*}
	From~\eqref{equ:vp}, we have
	\begin{equation*}
	\begin{aligned}
	\langle \partial_{t} \vpp, \partial_{x_{n}}Q \rangle
	&=\langle \Delta\vp-\vp+f'(Q)\vp,\partial_{x_{n}}Q\rangle-2\alpha\langle \vpp,\partial_{x_{n}}Q\rangle\\
	&+\langle G_{2},\partial_{x_{n}}Q\rangle
	+\langle R,\partial_{x_{n}}Q\rangle+\langle G,\partial_{x_{n}}Q\rangle+\langle \Mod_{2},\partial_{x_{n}}Q\rangle,
	\end{aligned}
	\end{equation*}
	where
	\begin{equation*}
	R=f(Q+V+\vp)-f(Q+V)-f'(Q)\vp.
	\end{equation*}
	Based on $-\Delta \partial_{x_{n}}q+\partial_{x_{n}}q-f'(q)\partial_{x_{n}}q=0$, integration by parts and~\eqref{equ:orthvpp},
	\begin{equation*}
	\langle \Delta\vp-\vp+f'(Q)\vp,\partial_{x_{n}}Q\rangle-2\alpha\langle \vpp,\partial_{x_{n}}Q\rangle=0.
	\end{equation*}
	Then, by the expression of $G_{2}$,~\eqref{small} and~\eqref{equ:zta},
	\begin{equation*}
	\begin{aligned}
	\left|\langle G_{2}, \partial_{x_{n}}Q\rangle\right|
	&\lesssim \left(|\dot{z}-\ell|+|\dot{\theta}-\beta|+|\ell|+|\beta|\right)\left(|\ell|+|\beta|+|b|\right)\\
	&\lesssim \mathcal{N}^{2}\left(\mathcal{N}+|a|+1\right)\lesssim \mathcal{N}^{2}.
	\end{aligned}
	\end{equation*}
		Next, using~\eqref{est:tay1},~\eqref{est:tay3} and~\eqref{small},
	\begin{equation*}
	\begin{aligned}
	\left|R\right|
	&\lesssim \left|f(Q+V+\vp)-f(Q+V)-f'(Q+V)\vp\right|+\left|(f'(Q+V)-f'(Q))\vp\right|\\
	&\lesssim |\vp|^{2}+|\vp|^{p}+(|V|+|V|^{p-1})|\vp|
	\lesssim |\vp|^{2}+|\vp|^{p}+|a|\left(|\Phi|+|\Phi|^{p-1}\right)|\vp|.
	\end{aligned}
	\end{equation*}
	It follows that 
	\begin{equation*}
	\begin{aligned}
	\left|\langle R,\partial_{x_{n}}Q\rangle\right|
	&\lesssim \int_{\RR^{N}}\left(|\vp|^{2}+|\vp|^{p}+|\vp||V|\right)\d x\lesssim \mathcal{N}^{2}+|a|\mathcal{N}.
	\end{aligned}
	\end{equation*}
	For $q$ be a non-degenerate state, we have $a=0$ which implies $G=0$. For $q$ be a degree-1 excited state, from~\eqref{est:tay6}, we have 
	\begin{equation*}
	\left|G-\frac{1}{2}f''(Q)V^{2}\right|\lesssim |V|^{3}+|V|^{p}\lesssim \left(|a|^{3}+|a|^{p}\right)\left(|\Phi|^{3}+|\Phi|^{p}\right).
	\end{equation*}
	It follows that 
	\begin{align*}
	&\left|\langle G,\partial_{x_{n}}Q\rangle-\langle \frac{1}{2}f''(Q)V^{2}, \partial_{x_{n}}Q\rangle\right|\\
	&\lesssim \left(|a|^{3}+|a|^{p}\right)\int_{\RR^{N}}\left(|\Phi|^{3}+|\Phi|^{p}\right)|\partial_{x_{n}}Q|\d x\lesssim |a|^{3}+|a|^{p}.
	\end{align*}
	Note that, by~\eqref{equ:can},~\eqref{equ:kert} and change of variables,
	\begin{equation*}
	\langle \frac{1}{2}f''(Q)V^{2},\partial_{x_{n}}Q\rangle=0.
	\end{equation*} 
	Then, from~\eqref{equ:phi},~\eqref{equ:kert},~\eqref{est:kernel},~\eqref{small}, the expression of $\Mod_{2}$ and change of variables,
	\begin{equation*}
	\begin{aligned}
	&\langle \Mod_{2},\partial_{x_{n}}Q\rangle\\
	&=\langle (\dot{\ell}+2\alpha\ell)\cdot \nabla Q, \partial_{x_{n}}Q\rangle 
	-\sum_{(i,j)\in I_{q}}\left(\dot{\beta}_{ij}+2\alpha \beta_{ij}\right)\langle \Psi_{ij},\partial_{x_{n}}Q\rangle-(\dot{b}+2\alpha b)\langle \Phi,\partial_{x_{n}}Q\rangle\\
	&=\langle (\dot{\ell}+2\alpha\ell)\cdot\nabla q,\partial_{x_{n}}q\rangle
	-\sum_{(i,j)\in I_{q}}\left(\dot{\beta}_{ij}+2\alpha\beta_{ij}\right)\langle \Omega_{ij}q,\partial_{x_{n}}q\rangle\\
	&\quad +O\left(\gamma\left(|\dot{\ell}+2\alpha\ell|+|\dot{\beta}+2\alpha\beta|\right)\right).
	\end{aligned}
	\end{equation*}
	Last, from~\eqref{small},~\eqref{equ:zta} and~\eqref{equ:dtxq},
	\begin{equation*}
	\begin{aligned}
	\left|\langle \vpp,\partial_{t}\partial_{x_{n}}Q\rangle\right|
	&\lesssim \|\vec{\varphi}\|_{\E}\left(|\dot{z}-\ell|+|\ell|+|\dot{\theta}-\beta|+|\beta|\right)\lesssim \mathcal{N}^{2}\left(\mathcal{N}+|a|+1\right)\lesssim \mathcal{N}^{2}.
	\end{aligned}
	\end{equation*}
	In conclusion of the previous estimates, the orthogonality condition $\langle \vpp, \partial_{x_{n}}Q\rangle=0$ gives the following,
	\begin{equation*}
	\begin{aligned}
	&\langle (\dot{\ell}+2\alpha\ell)\cdot\nabla q,\partial_{x_{n}}q\rangle
	-\sum_{(i,j)\in I_{q}}\left(\dot{\beta}_{ij}+2\alpha\beta_{ij}\right)\langle \Omega_{ij}q,\partial_{x_{n}}q\rangle\\
	&+O\left(\gamma\left(|\dot{\ell}+2\alpha\ell|+|\dot{\beta}+2\alpha\beta|\right)\right)=O\left(\mathcal{N}^{2}+|a|\mathcal{N}+|a|^{3}+|a|^{p}\right).
	\end{aligned}
	\end{equation*}
	Proceeding similarly for the orthogonality conditions $\langle \phi_2, \Omega_{ij} Q \rangle =0$ in~\eqref{equ:orthvpp} and using again the fact that family $\left\{\partial_{x_{n}}q,1\le n\le N;\Omega_{ij}q,(i,j)\in I_{q}\right\}$ is linearly independent, we find~\eqref{equ:bl} for $\gamma$ small enough.
	
	\bigskip
	
	Proof of (ii). We prove~\eqref{est:ba2}; the proof of~\eqref{est:a} is same as~\eqref{equ:zta}.
	We differentiate the orthogonality $\langle \vpp, \Phi\rangle=0 $ in~\eqref{equ:orthvpp},
	\begin{equation*}
		\begin{aligned}
			0=&\frac{\d}{\d t}\langle\vpp,\Phi \rangle=\langle \partial_{t} \vpp,\Phi\rangle+\langle \vpp,\partial_{t}\Phi\rangle.
		\end{aligned}
	\end{equation*}
	Using again~\eqref{equ:vp},
	\begin{equation*}
		\begin{aligned}
			\langle \partial_{t} \vpp,\Phi\rangle
			&=\langle \Delta\vp-\vp+f'(Q)\vp,\Phi\rangle-2\alpha\langle \vpp,\Phi\rangle\\
			&+\langle G_{2},\Phi\rangle+\langle R,\Phi\rangle+\langle G+ \Mod_{2},\Phi\rangle,
		\end{aligned}
	\end{equation*}
	where
	\begin{equation*}
		R=f(Q+V+\vp)-f(Q+V)-f'(Q)\vp.
	\end{equation*}
	From $-\Delta \Phi+\Phi-f'(Q)\Phi=0$, integration by parts and~\eqref{equ:orthvpp},
	\begin{equation*}
		\langle \Delta\vp-\vp+f'(Q)\vp,\Phi\rangle-2\alpha\langle \vpp,\Phi\rangle=0.
	\end{equation*}
	By the expression of $G_{2}$,~\eqref{small} and~\eqref{equ:zta},
	\begin{equation*}
	\begin{aligned}
		\left|\langle G_{2}, \Phi\rangle\right|
		&\lesssim \left(|\dot{z}-\ell|+|\dot{\theta}-\beta|+|\ell|+|\beta|\right)\left(|\ell|+|\beta|+|b|\right)\lesssim \mathcal{N}^{2}.
		\end{aligned}
	\end{equation*}
	Recall that, from~\eqref{est:tay1},~\eqref{est:tay3} and~\eqref{small},
	\begin{equation*}
	\begin{aligned}
	\left|R\right|
	\lesssim |\vp|^{2}+|\vp|^{p}+|a|\left(|\Phi|+|\Phi|^{p-1}\right)|\vp|.
	\end{aligned}
	\end{equation*}
	Based on above inequality and the Sobolev embedding Theorem,
	\begin{align*}
		\left|\langle R,\Phi\rangle\right|
		&\lesssim \int_{\RR^{N}}\left(|\vp|^{2}+|\vp|^{p}+|a|\left(|\Phi|+|\Phi|^{p-1}\right)|\vp|\right)|\Phi|\d x\\
		&\lesssim \|\vec{\varphi}\|^{2}_{\E}+\|\vec{\varphi}\|^{p}_{\E}+|a|\|\vec{\varphi}\|_{\E}\lesssim \mathcal{N}^{2}+|a|\mathcal{N}.
	\end{align*}
	
	Then, from~\eqref{est:tay6} and $p>2$,
	\begin{equation*}
	\left|G-\frac{1}{2}f''(Q)V^{2}\right|\lesssim |V|^{3}+|V|^{p}\lesssim \left(|a|^{3}+|a|^{p}\right)\left(|\Phi|^{3}+|\Phi|^{p}\right).
	\end{equation*}
It follows that 
	\begin{align*}
	\left|\langle G-\frac{1}{2}f''(Q)V^{2},\Phi\rangle\right|
	&\lesssim \left(|a|^{3}+|a|^{p}\right)\int_{\RR^{N}}\left(|\Phi|^{4}+|\Phi|^{p+1}\right)\d x\lesssim |a|^{3}+|a|^{p}.
	\end{align*}
    Note that, by our normalization choice for $\phi$ \eqref{def:kappa}, and a change of variables,
    \begin{equation*}
    \langle \frac{1}{2}f''(Q)V^{2},\Phi\rangle=\frac{a^{2}}{2}E'''(q) (\phi,\phi,\phi)=-2\alpha a^{2} \| \phi \|^2.
    \end{equation*} 
Thus, from~\eqref{equ:phi},~\eqref{equ:kert}, the expression of $\Mod_{2}$, the above estimates and change of variables,
	\begin{equation*}
		\begin{aligned}
			&\langle \Mod_{2}+G,\Phi\rangle\\
			&=-(\dot{b}+2\alpha b)\langle \Phi,\Phi\rangle+\langle \frac{1}{2}f''(Q)V^{2},\Phi\rangle+
			\langle G-\frac{1}{2}f''(Q)V^{2},\Phi\rangle \\
			&=-(\dot{b}+2\alpha b+2\alpha a^{2}) \langle \phi,\phi\rangle+O\left(|a|^{3}+|a|^{p}\right).
		\end{aligned}
	\end{equation*}
	Last, from~\eqref{equ:dtQPhiY},~\eqref{small},~\eqref{equ:zta} and the Sobolev embedding Theorem,
	\begin{equation*}
		\begin{aligned}
			\left|\langle \vpp,\partial_{t}\Phi\rangle\right|
			&\lesssim \|\vec{\varphi}\|_{\E}\left(|\dot{z}-\ell|+|\dot{\theta}-\beta|+|\ell|+|\beta|\right)\lesssim \mathcal{N}^{2}\left(\mathcal{N}+|a|+1\right)\lesssim \mathcal{N}^{2}.
		\end{aligned}
	\end{equation*}
	Combining the above estimates, we have 
	\begin{equation*}
		\begin{aligned}
			&(\dot{b}+2\alpha b+2\alpha a^{2})\langle \phi,\phi\rangle=O\left(\mathcal{N}^{2}+|a|\mathcal{N}+|a|^{3}+|a|^{p}\right),
		\end{aligned}
	\end{equation*}
	which means~\eqref{est:ba2}.
	\end{proof}

Last, we derive the control of exponential directions.
\begin{lemma}\label{le:equexp} 
	Let $a_{k}^{\pm}=\langle \vec{\varphi},\vec{Z}^{\pm}_{k}\rangle$ for $k=1,\cdots,K$. Then
\begin{equation}\label{equ:exp}
\left|\frac{\d}{\d t}a_{k}^{\pm}- \nu^{\pm}_{k}a_{k}^{\pm}\right|\lesssim  \mathcal{N}^{2}+a^{2}.
\end{equation}
\end{lemma}
\begin{proof}
	 By~\eqref{equ:vp},
	\begin{equation*}
	\begin{aligned}
	\frac{\d}{\d t}a_{k}^{\pm}=&\langle \partial_{t}\vec{\varphi},\vec{Z}_{k}^{\pm}\rangle
	+\langle \vec{\varphi},\partial_{t}\vec{Z}_{k}^{\pm}\rangle\\
	=&\langle\Delta \vp-\vp+f'(Q)\vp,\Upsilon_{k}\rangle+\left(\zeta_{k}^{\pm}-2\alpha\right)\langle \vpp,\Upsilon_{k}\rangle+ \langle R+G,\Upsilon_{k}\rangle\\
	&+\langle \zeta_{k}^{\pm}\Mod_{1}+\Mod_{2},\Upsilon_{k}\rangle+\langle \zeta_{k}^{\pm}G_{1}+G_{2},\Upsilon_{k}\rangle+\langle \vec{\varphi},\partial_{t}\vec{Z}_{k}^{\pm}\rangle.
	\end{aligned}
	\end{equation*}
	From $\lambda_{k}^{2}=\nu_{k}^{\pm}\zeta_{k}^{\pm}$, $\mathcal{L}_{q}Y_{k}=-\lambda_{k}^{2}Y_{k}$ and integration by parts,
	\begin{equation*}
	\langle\Delta \vp-\vp+f'(Q)\vp,\Upsilon_{k}\rangle
	=\langle\vp,\left(\Delta-1+f'(Q)\right)\Upsilon_{k}\rangle
	=\nu_{k}^{\pm}\zeta_{k}^{\pm}\langle \vp, \Upsilon_{k}\rangle.
	\end{equation*}
	Combining above identity with $\nu_{k}^{\pm}=\zeta_{k}^{\pm}-2\alpha$ and the definition of $a_{k}^{\pm}$, 
	\begin{equation*}
	\begin{aligned}
	&\langle\Delta \vp-\vp+f'(Q)\vp,\Upsilon_{k}\rangle+\left(\zeta_{k}^{\pm}-2\alpha\right)\langle \vpp,\Upsilon_{k}\rangle\\
	&=\nu_{k}^{\pm}\left(\zeta_{k}^{\pm}\langle \vp, \Upsilon_{k}\rangle+\langle \vpp,\Upsilon_{k}\rangle\right)=\nu_{k}^{\pm}a_{k}^{\pm}.
	\end{aligned}
	\end{equation*}
	Recall that, from~\eqref{est:tay1},~\eqref{est:tay2} and~\eqref{est:tay3},
	\begin{equation*}
	|R|+|G|\lesssim |\vp|^{2}+|\vp|^{p}+|V|^{2}+|V|^{p}.
	\end{equation*}
	Therefore, by the Sobolev embedding Theorem, and as $\Upsilon_k$ is exponentially localized, we have 
	\begin{equation*}
	\left|\langle R+G,\Upsilon_{k} \rangle\right|\lesssim \int_{\RR^{N}}\left(\vp^{2}+|\vp|^{p}+|V|^{2}+|V|^{p}\right)\left|\Upsilon_{k}\right|\d x \lesssim \mathcal{N}^{2}+a^{2}.
	\end{equation*}
	Note that, from $\langle \psi,Y_{k}\rangle=0$ for any $\psi\in \ker\mathcal{L}_{q}$ and $k=1,\cdots,K$, 
	\begin{equation*}
	\langle \zeta_{k}^{\pm}\Mod_{1}+\Mod_{2},\Upsilon_{k}\rangle=0.
	\end{equation*}
	Next, by the expression of~$G_{1}$,~$G_{2}$,~\eqref{small},~\eqref{equ:zta} and~\eqref{equ:bl},
	\begin{equation*}
	\begin{aligned}
	\left|\langle \zeta_{k}^{\pm}G_{1}+G_{2},\Upsilon_{k}\rangle\right|
	&\lesssim \big(|\dot{z}-\ell|+|\dot{\theta}-\beta|+|\ell|+|\beta|\big)\left(|\ell|+|\beta|+|a|+|b|\right)\\
	&\lesssim \mathcal{N}
	\left(\mathcal{N}+|a|\right)\left(\mathcal{N}+|a|+1\right)\lesssim \mathcal{N}^{2}+a^{2}.
	\end{aligned}
	\end{equation*}
	Last, from~\eqref{equ:dtQPhiY}~\eqref{small},~\eqref{equ:zta} and the Sobolev embedding Theorem,
	\begin{equation*}
	\begin{aligned}
	\big|\langle \vec{\varphi},\partial_{t}\vec{Z}_{k}^{\pm}\rangle\big|
	&\lesssim \int_{\RR^{N}}\left(|\vp|+|\vpp|\right)|\partial_{t} \Upsilon_{k}|\d x\\
	&\lesssim  \|\vec{\varphi}\|_{\E}\left(|\dot{z}-\ell|+|\dot{\theta}-\beta|+|\ell|+|\beta|\right)\lesssim \mathcal{N}^{2}\left(\mathcal{N}+|a|+1\right)\lesssim \mathcal{N}^{2}.
	\end{aligned}
	\end{equation*}
	Gathering these estimates, and proceeding similarly for all $k=1,\cdots,K$, 
	we obtain~\eqref{equ:exp}.
\end{proof}

\subsection{Energy estimates} 

Let $q$ be a non-degenerate or degree-1 excited state. For $\mu>0$ small to be chosen later, we denote $\rho=2\alpha-\mu$, and consider the nonlinear energy functional
\begin{equation}\label{def:E}
\begin{aligned}
\mathcal{E}=&\int_{\RR^{N}}\left\{|\nabla \varphi_{1}|^{2}+(1-\rho\mu)\varphi_{1}^{2}+(\varphi_{2}+\mu \varphi_{1})^{2}\right\}\d x\\
&-2\int_{\RR^{N}}\left\{F(Q+V+\varphi_{1})-F(Q+V)-f(Q+V)\varphi_{1}\right\}\d x.
\end{aligned}
\end{equation}
Analogous energy functionals were introduced in \cite{CMYZ, CMY}, in order to take full advantage of the damping of \eqref{equDKG}, as is shown in the Lemma below.
Recall that in the previous paragraph, in \eqref{def:N}, we set
\begin{equation*}
\mathcal{N}=\|\vec{\varphi}\|_{\E}+|\ell|+|\beta|+|b|.
\end{equation*}

\begin{lemma}\label{le:ener1}
	In the context of Proposition~\ref{pro:dec}, there exist $0<\gamma_{0}\ll 1$ and $0<\mu<\min_{k,\pm}(1,\alpha,\nu_{k}^{\pm})$ such that the following hold, for all $t \in [T_1,T_2]$.
	\begin{enumerate}
		\item {\rm{Coercivity and bound.}}
		\begin{equation}\label{coer:E}
		\mu\|\vec{\varphi}\|_{\E}^{2}
		-\mu^{-1}\sum_{k=1}^{K}\left((a_{k}^{+})^{2}+(a_{k}^{-})^{2}\right)\le \mathcal{E}\le \mu^{-1}\|\vec{\varphi}\|_{\E}^{2}.
		\end{equation}
\item {\rm{Time variation}.}
\begin{equation}\label{dtE}
\frac{\d}{\d t}\mathcal{E}+2\mu \mathcal{E}\le \mu^{-1}\left(\mathcal{N}^{3}+a^{2}\mathcal{N}\right).
\end{equation}
\end{enumerate}
\end{lemma}
\begin{proof}
	Proof of (i). First, from~\eqref{est:tay4} and the Sobolev embedding Theorem,
	\begin{equation*}
	\begin{aligned}
	&\int_{\RR^{N}}\left|F(Q+V+\vp)-F(Q+V)-f(Q+V)\vp\right|\d x\\
	&\lesssim\int_{\RR^{N}}(|Q+V|^{p-1}|\vp|^{2}+|\vp|^{p+1})\d x\lesssim\|\vec{\varphi}\|_{\E}^{2}+\|\vec{\varphi}\|_{\E}^{p+1}.
	\end{aligned}
	\end{equation*}
	It follows that the right-hand side of~\eqref{coer:E}. 
	
	Second, from the definition of $\mathcal{E}$ in~\eqref{def:E}, we decompose
	\begin{equation*}
	\mathcal{E}
	=\langle -\Delta\vp+\vp-f'(Q)\vp,\vp\rangle-\rho\mu \langle \vp,\vp\rangle
	+\mathcal{E}_{1}+\mathcal{E}_{2}+\mathcal{E}_{3},
	\end{equation*}
where
\begin{equation*}
\begin{aligned}
\mathcal{E}_{1}&=\int_{\RR^{N}}(\vpp+\mu\vp)^{2}\d x,\quad \mathcal{E}_{2}=-\int_{\RR^{N}}\left(f'(Q+V)-f'(Q)\right)\vp^{2}\d x,\\
\mathcal{E}_{3}&=-2\int_{\RR^{N}}\big\{F(Q+V+\vp)-F(Q+V)-f(Q+V)\vp-\frac{1}{2}f'(Q+V)\vp^{2}\big\}\d x.
\end{aligned}
\end{equation*}
	By~\eqref{coer} and~\eqref{equ:orthvp},
	\begin{equation*}
	\langle -\Delta\vp+\vp-f'(Q)\vp,\vp\rangle \ge 
	c\|\vp\|_{H^{1}}^{2}-c^{-1}\sum_{k=1}^{K}\left((a_{k}^{+})^{2}+(a^{-}_{k})^{2}\right).
	\end{equation*}
	From the AM-GM inequality, we have 
	\begin{equation*}
	\mathcal{E}_{1}=\int_{\RR^{N}}\left(\vpp+\mu \vp\right)^{2}\d x\ge \frac{1}{2}\int_{\RR^{N}}\vpp^{2}\d x-\mu^{2}\int_{\RR^{N}}\vp^{2}\d x.
	\end{equation*}
    Then, using~\eqref{est:tay1},~\eqref{est:tay5},~\eqref{small} and the Sobolev embedding Theorem, we have 
    \begin{equation*}
    \begin{aligned}
    \left|\mathcal{E}_{2}\right|&\lesssim
    \int_{\RR^{N}}\left(|V|+|V|^{p-1}\right)|\vp|^{2}\d x\lesssim 
    \left(\gamma_{0}+\gamma_{0}^{p-1}\right)\|\vec{\varphi}\|_{\E}^{2},\\
\left|\mathcal{E}_{3}\right|&\lesssim \int_{\RR^{N}}\left(|\vp|^{3}+|\vp|^{p+1}\right)\d x\lesssim \|\vec{\varphi}\|_{\E}^{3}+\|\vec{\varphi}\|_{\E}^{p+1}.
\end{aligned}
\end{equation*}
Combining the above estimates, we have
\begin{equation*}
\begin{aligned}
\mathcal{E}\ge& \left(c-\rho\mu-\mu^{2}\right)\|\vp\|_{H^{1}}^{2}+\frac{1}{2}\|\vpp\|_{L^{2}}^{2}-
c^{-1}\sum_{k=1}^{K}\left( (a_{k}^{+})^{2}+(a^{-}_{k})^{2}\right)\\
&+O\left(\left(\gamma_{0}+\gamma_{0}^{p-1}\right)\|\vec{\varphi}\|_{\E}^{2}+\|\vec{\varphi}\|_{\E}^{3}\right),
\end{aligned}
\end{equation*} 
which implies the left-hand side of~\eqref{coer:E} for $\gamma_{0}>0$ small enough and $\mu>0$ small enough.

	Proof of (ii). By direct computation and integration by parts, we decompose,
	\begin{equation*}
	\begin{aligned}
	\frac{\d}{\d t}\mathcal{E}=
	&2\int_{\RR^{N}}\partial_{t}\vp\big\{-\Delta \vp+(1-\rho\mu)\vp-\big[f(Q+V+\vp)-f(Q+V)\big]\big\}\d x\\
	&-2\int_{\RR^{N}}\left\{\left(f(Q+V+\vp)-f(Q+V)-f'(Q+V)\vp\right)\left(\partial_{t} Q+\partial_{t} V\right)\right\}\d x\\
	&+2\int_{\RR^{N}}\left\{(\vpp+\mu\vp)(\partial_{t}\vpp+\mu\partial_{t}\vp)\right\}\d x=\mathcal{I}_{1}+\mathcal{I}_{2}+\mathcal{I}_{3}.
	\end{aligned}
	\end{equation*}

	\emph{Estimate on ${\mathcal{I}}_{1}$}.
	We claim
	\begin{equation}\label{est:I1}
	\begin{aligned}
	\mathcal{I}_{1}
	=&-2\int_{\RR^{N}}\left\{\Delta \vp-\vp+\left[f(Q+V+\vp)-f(Q+V)\right]\right\}\vpp\d x\\
	&-2\rho \mu\int_{\RR^{N}}\vp \vpp \d x+O\left(\mathcal{N}^{3}+|a|\mathcal{N}^{2}\right).
	\end{aligned}
	\end{equation}
	By~\eqref{equ:orthvp},~\eqref{equ:vp}, $-\Delta\Mod_{1}+\Mod_{1}-f'(Q)\Mod_{1}=0$ and integration by parts,
	\begin{equation*}
	\begin{aligned}
	\mathcal{I}_{1}
	=&2\int_{\RR^{N}}\left\{-\Delta \vp+(1-\rho\mu)\vp-\left[f(Q+V+\vp)-f(Q+V)\right]\right\}\vpp\d x\\
	&+\mathcal{I}_{1,1}+\mathcal{I}_{1,2}+\mathcal{I}_{1,3}+\mathcal{I}_{1,4},
	\end{aligned}
	\end{equation*}
where
\begin{equation*}
\begin{aligned}
\mathcal{I}_{1,1}&=-2\int_{\RR^{N}}\left[\Delta \vp-(1-\rho\mu)\vp\right] G_{1}\d x,\\
\mathcal{I}_{1,2}&=-2\int_{\RR^{N}}\left[f(Q+V+\vp)-f(Q+V)\right]G_{1}\d x,\\
\mathcal{I}_{1,3}&=-2\int_{\RR^{N}}\left[\left(f'(Q+V)-f'(Q)\right)\vp\right]\Mod_{1}\d x,\\
\mathcal{I}_{1,4}&=-2\int_{\RR^{N}}\left[f(Q+V+\vp)-f(Q+V)-f'(Q+V)\vp\right]\Mod_{1}\d x.
\end{aligned}
\end{equation*}
	Using~\eqref{est:tay3},~\eqref{small},~\eqref{equ:zta} and integration by parts, we have 
	\begin{equation*}
	\begin{aligned}
	\left|\mathcal{I}_{1,1}\right|
	&\lesssim  \int_{\RR^{N}}\left(\left|\Delta G_{1}\right|+|G_{1}|\right)|\vp|\d x\\
	&\lesssim   |a|\mathcal{N}\left(|\dot{z}-\ell|+|\dot{\theta}-\beta|+\mathcal{N}\right)\lesssim |a|\mathcal{N}^{2}\left(\mathcal{N}+|a|+1\right)\lesssim |a|\mathcal{N}^{2},\\
	\left|\mathcal{I}_{1,2}\right|
	&\lesssim \int_{\RR^{N}}|G_{1}|(|\vp|+|\vp|^{p})\d x\\
	&\lesssim   |a|\mathcal{N}\left(|\dot{z}-\ell|+|\dot{\theta}-\beta|+\mathcal{N}\right)\lesssim |a|\mathcal{N}^{2}\left(\mathcal{N}+|a|+1\right)\lesssim |a|\mathcal{N}^{2}.
	\end{aligned}
	\end{equation*}
	Then, using~\eqref{est:tay1},~\eqref{est:tay3},~\eqref{small} and~\eqref{equ:zta}, we have
	\begin{equation*}
	\begin{aligned}
	\left|\mathcal{I}_{1,3}\right|
	&\lesssim \int_{\RR^{N}}\left(|V|+|V|^{p-1}\right)|\vp||\Mod_{1}|\d x\\
	&\lesssim \|\vec{\varphi}\|_{\E}\left(|a|+|a|^{p-1}\right)\left(|\dot{z}-\ell|+|\dot{\theta}-\beta|+|\dot{a}-b|\right)\\
	&\lesssim \mathcal{N}^{2}\left(|a|+|a|^{p-1}\right)\left(\mathcal{N}+|a|\right)
	\lesssim \mathcal{N}^{3}+|a|\mathcal{N}^{2},
	\end{aligned}
	\end{equation*}
and
\begin{equation*}
\begin{aligned}
\left|\mathcal{I}_{1,4}\right|
&\lesssim \int_{\RR^{N}}\left(\vp^{2}+|\vp|^{p}\right)|\Mod_{1}|\d x\\
&\lesssim \left(\mathcal{N}^{2}+\mathcal{N}^{p}\right)\left(|\dot{z}-\ell|+|\dot{\theta}-\beta|+|\dot{a}-b|\right)\lesssim \mathcal{N}\left(\mathcal{N}^{2}+\mathcal{N}^{p}\right)\left(\mathcal{N}+|a|\right)
\lesssim \mathcal{N}^{3}.
\end{aligned}
\end{equation*}
	Gathering the above estimates, we obtain~\eqref{est:I1}.
	
\emph{Estimate on ${\mathcal{I}}_{2}$.} We claim
\begin{equation}\label{est:I2}
\left|\mathcal{I}_{2}\right|\lesssim \mathcal{N}^{3}.
\end{equation}
By~\eqref{equ:dtQPhiY},
\begin{equation*}
\partial_{t}Q+\partial_{t}V=-\ell\cdot\nabla Q+\sum_{(i,j)\in I_{q}}\beta_{ij}\Psi_{ij}+b\Phi-{\rm{Mod}_{1}}-G_{1}.
\end{equation*}
Based on above identity, we decompose
\begin{equation*}
\mathcal{I}_{2}=\mathcal{I}_{2,1}+\mathcal{I}_{2,2}+\mathcal{I}_{2,3}+\mathcal{I}_{2,4}+\mathcal{I}_{2,5},
\end{equation*}
where
\begin{equation*}
\begin{aligned}
\mathcal{I}_{2,1}&=2\int_{\RR^{N}}\left(f(Q+V+\vp)-f(Q+V)-f'(Q+V)\vp\right)\left(-b\Phi\right) \d x,\\
\mathcal{I}_{2,2}&=2\int_{\RR^{N}}\left(f(Q+V+\vp)-f(Q+V)-f'(Q+V)\vp\right)G_{1}\d x,
\end{aligned}
\end{equation*}
and
\begin{equation*}
\begin{aligned}
\mathcal{I}_{2,3}&=2\int_{\RR^{N}}\left(f(Q+V+\vp)-f(Q+V)-f'(Q+V)\vp\right)\Mod_{1}\d x,\\
\mathcal{I}_{2,4}&=2\int_{\RR^{N}}\left(f(Q+V+\vp)-f(Q+V)-f'(Q+V)\vp\right)\big(\ell\cdot\nabla Q\big)\d x,\\
\mathcal{I}_{2,5}&=-2\sum_{(i,j)\in I_{q}}\beta_{ij}\int_{\RR^{N}}\left(f(Q+V+\vp)-f(Q+V)-f'(Q+V)\vp\right)
\Psi_{ij}\d x.
\end{aligned}
\end{equation*}
From~\eqref{est:tay3} and the Sobolev embedding Theorem,
\begin{equation*}
\begin{aligned}
\left|\mathcal{I}_{2,1}\right|\lesssim |b|\int_{\RR^{N}}\vp^{2}|\Phi|\d x\lesssim \mathcal{N}^{3}.
\end{aligned}
\end{equation*}
Using~\eqref{est:tay3},~\eqref{small},~\eqref{equ:zta},~\eqref{est:a} and the Sobolev embedding Theorem,
we have 
\begin{equation*}
\begin{aligned}
\left|\mathcal{I}_{2,2}\right|
&\lesssim \int_{\RR^{N}}\left(\vp^{2}+|\vp|^{p}\right)|G_{1}|\d x\\
&\lesssim \left(\mathcal{N}^{2}+\mathcal{N}^{p}\right)\left(|a|\mathcal{N}+|a||\dot{z}-\ell|+|a||\dot{\theta}-\beta|\right)\\
&\lesssim \left(\mathcal{N}^{2}+\mathcal{N}^{p}\right)\left(|a|\mathcal{N}+|a|\mathcal{N}(|a|+\mathcal{N})\right)\lesssim \mathcal{N}^{3},
\end{aligned}
\end{equation*}
\begin{equation*}
\begin{aligned}
\left|\mathcal{I}_{2,3}\right|
&\lesssim \int_{\RR^{N}}\left(\vp^{2}+|\vp|^{p}\right)|\Mod_{1}|\d x\\
&\lesssim \left(\mathcal{N}^{2}+\mathcal{N}^{p}\right)\left(\dot{z}-\ell|+|\dot{\theta}-\beta|+|\dot{a}-b|\right)\\
&\lesssim \left(\mathcal{N}^{2}+\mathcal{N}^{p}\right)\left(|a|+\mathcal{N}\right)\mathcal{N}\lesssim\mathcal{N}^{3},
\end{aligned}
\end{equation*}
\begin{equation*}
\left|\mathcal{I}_{2,4}\right|
\lesssim \int_{\RR^{N}}\left(\vp^{2}+|\vp|^{p}\right)|\ell\cdot \nabla Q|\d x
\lesssim\mathcal{N}^{3},
\end{equation*}
and
\begin{equation*}
\begin{aligned}
\left|\mathcal{I}_{2,5}\right|\lesssim \sum_{(i,j)\in I_{q}}\int_{\RR^{N}}\left(\vp^{2}+|\vp|^{p}\right)\left|\beta_{ij}\Psi_{ij}\right|\d x\lesssim \mathcal{N}^{3}.
\end{aligned}
\end{equation*}
Combining above estimates with~\eqref{small}, we obtain~\eqref{est:I2}.

\emph{Estimate on $\mathcal{I}_{3}$.} We claim
\begin{equation}\label{est:I3}
\begin{aligned}
\mathcal{I}_{3}
=&-2\mu \mathcal{E}+2\int_{\RR^{N}}\left\{\Delta \vp-\vp+\left[f(Q+V+\vp)-f(Q+V)\right]\right\}\vpp\d x\\
&+2\rho\mu \int_{\RR^{N}}\vp \vpp \d x-2(\rho-\mu)\int_{\RR^{N}}(\vpp+\mu \vp)^{2}\d x+O\left(\mathcal{N}^{3}+a^{2}\mathcal{N}\right).
\end{aligned}
\end{equation}
By direct computation and~\eqref{equ:vp}, we decompose,
\begin{equation*}
\begin{aligned}
\mathcal{I}_{3}
=&-2\mu \mathcal{E}+2\int_{\RR^{N}}\left\{\Delta \vp-\vp+\left[f(Q+V+\vp)-f(Q+V)\right]\right\}\vpp\d x\\
&+2\rho\mu \int_{\RR^{N}}\vp \vpp \d x-2(\rho-\mu)\int_{\RR^{N}}(\vpp+\mu \vp)^{2}\d x\\
&+\mathcal{I}_{3,1}+\mathcal{I}_{3,2}
+\mathcal{I}_{3,3}+\mathcal{I}_{3,4}+\mathcal{I}_{3,5},
\end{aligned}
\end{equation*}
where
\begin{equation*}
\begin{aligned}
\mathcal{I}_{3,1}&=2\int_{\RR^{N}}\left(\vpp+\mu\vp\right)G\d x,\\
\mathcal{I}_{3,2}&=2\int_{\RR^{N}}\left(\vpp+\mu\vp\right)\left(G_{2}+\mu G_{1}\right)\d x,\\
\mathcal{I}_{3,3}&=2\int_{\RR^{N}}\left(\vpp+\mu\vp\right)\left(\Mod_{2}+\mu\Mod_{1}\right)\d x,
\end{aligned}
\end{equation*}
and
\begin{equation*}
\begin{aligned}
\mathcal{I}_{3,4}&=2\mu \int_{\RR^{N}}\left[f(Q+V+\vp)-f(Q+V)-f'(Q+V)\vp\right]\vp \d x,\\
\mathcal{I}_{3,5}&=-4\mu \int_{\RR^{N}}\big[F(Q+V+\vp)-F(Q+V)-f(Q+V)\vp-\frac{1}{2}f'(Q+V)\vp^{2}\big]\d x.
\end{aligned}
\end{equation*}
Using~\eqref{est:tay2},
\begin{equation*}
\left|\mathcal{I}_{3,1}\right|
\lesssim \int_{\RR^{N}}\left(|\vpp|+|\vp|\right)V^{2}\d x
\lesssim a^{2}\mathcal{N}.
\end{equation*}
Next, from~\eqref{small},~\eqref{equ:zta} and the AM-GM inequality,
\begin{equation*}
\begin{aligned}
\left|\mathcal{I}_{3,2}\right|
&\lesssim \int_{\RR^{N}}\left(|\vp|+|\vpp|\right)\left(|G_{1}|+|G_{2}|\right)\d x\\
&\lesssim \mathcal{N}\left(|a|+|b|+|\ell|+|\beta|\right)\left(|\dot{z}-\ell|+|\dot{\theta}-\beta|+|\ell|+|\beta|\right)\\
&\lesssim \mathcal{N}^{2}\left(|a|+\mathcal{N}\right)\left(\mathcal{N}+|a|+1\right)\lesssim a^{2}\mathcal{N}+\mathcal{N}^{3}.
\end{aligned}
\end{equation*}
Note that, from~\eqref{equ:orthvp} and~\eqref{equ:orthvpp}, we have
\begin{equation*}
\mathcal{I}_{3,3}=2\langle \vpp, \Mod_{2}+\mu\Mod_{1}\rangle+2\mu\langle \vp, \Mod_{2}+\mu\Mod_{1}\rangle=0.
\end{equation*}
Last, from~\eqref{est:tay3},~\eqref{est:tay5} and the Sobolev embedding Theorem,
\begin{equation*}
\begin{aligned}
\left|\mathcal{I}_{3,4}\right|
&\lesssim \int_{\RR^{N}}\left(|\vp|^{2}+|\vp|^{p}\right)|\vp|\d x
\lesssim \mathcal{N}^{3}+\mathcal{N}^{p+1},\\
\left|\mathcal{I}_{3,5}\right|
&\lesssim \int_{\RR^{N}}\left(|\vp|^{3}+|\vp|^{p+1}\right)\d x
\lesssim \mathcal{N}^{3}+\mathcal{N}^{p+1}.
\end{aligned}
\end{equation*}
Gathering the above estimates, we obtain~\eqref{est:I3}. 

Last, combining~\eqref{est:I1},~\eqref{est:I2} and~\eqref{est:I3}, we conclude~\eqref{dtE}.
	\end{proof}
We observe from the definition and the energy property~\eqref{energy} that a packed solution $\vec{u}$ with degree-1 excited state cluster point $q$ satisfies
\begin{equation}\label{lim:energy}
\lim_{t\to \infty}E(\vec{u}(t))=\frac{1}{2}\int_{\RR^{N}}\left(|\nabla q|^{2}+q^{2}-2F(q)\right)\d x=E(q,0).
\end{equation}
More precisely, we expand the energy for a solution close to degree-1 excited state.
	\begin{lemma}\label{le:expanE}
		In the context of Proposition~\ref{pro:dec}, for the case $q$ be a degree-1 excited state, we have 
		\begin{equation}\label{est:Eua3}
	E(\vec{u})
	=E(q,0)+\frac{2}{3}\alpha\|\phi\|_{L^{2}}^{2}a^{3}+O\left(\mathcal{N}^{2}+|a|^{\bar{p}+1}\right).
		\end{equation}
	\end{lemma}
\begin{proof}
	By~\eqref{def:vp} and an elementary computation, we find
	\begin{equation*}
	\begin{aligned}
	|u|^{2}&=|Q|^{2}+|V|^{2}+|\vp|^{2}+2Q(V+\vp)+2V\vp,\\
	|\nabla u|^{2}&=|\nabla Q|^{2}+|\nabla V|^{2}+|\nabla \vp|^{2}+2\nabla Q\cdot\nabla (V+\vp)+2\nabla V\cdot\nabla \vp,
	\end{aligned}
	\end{equation*}
	and
	\begin{equation*}
	\begin{aligned}
	F(u)
	=&F(Q)+f(Q)(V+\vp)+\frac{1}{2}f'(Q)(V^{2}+\vp^{2}+2V\vp)+\frac{1}{6}f''(Q)V^{3}\\
	&+F(Q+V+\vp)-F(Q)-f(Q)(V+\vp)-\frac{1}{2}f'(Q)(V+\vp)^{2}\\
	&-\frac{1}{6}f''(Q)(V+\vp)^{3}+\frac{1}{6}f''(Q)\left((V+\vp)^{3}-V^{3}\right).
	\end{aligned}
	\end{equation*}
	Based on the above identities and change of variables, we have 
	\begin{equation*}
	E(\vec{u})=E(q,0)-\frac{1}{6}\int_{\RR^{N}}f''(Q)V^{3}\d x+E_{1}+E_{2}+E_{3}+E_{4}+E_{5}+E_{6},
	\end{equation*}
	where
	\begin{equation*}
	\begin{aligned}
	E_{1}&=\frac{1}{2}\int_{\RR^{N}}\left(|\nabla V|^{2}+|V|^{2}-f'(Q)V^{2}\right)\d x,\\
	E_{2}&=\int_{\RR^{N}}\left(\nabla V\cdot\nabla \vp+V\vp-f'(Q)V\vp\right)\d x,\\
	E_{3}
	&=\frac{1}{2}\int_{\RR^{N}}\left(|\nabla \vp|^{2}+|\vp|^{2}-f'(Q)\vp^{2}+(\partial_{t}u)^{2}\right)\d x,\\
	E_{4}&=\int_{\RR^{N}}\left(\nabla Q\cdot\nabla (V+\vp)+Q(V+\vp)-f(Q)(V+\vp)\right)\d x,
	\end{aligned}
	\end{equation*}
	and
	\begin{equation*}
	\begin{aligned}
	E_{5}&=-\frac{1}{3}\int_{\RR^{N}}f''(Q)\left((V+\vp)^{3}-V^{3}\right)\d x,\\
	E_{6}&=-2\int_{\RR^{N}}\bigg(F(Q+V+\vp)-F(Q)-f(Q)(V+\vp)\\
	&\quad \quad \quad \quad \quad \quad 
	-\frac{1}{2}f'(Q)(V+\vp)^{2}-\frac{1}{6}f''(Q)(V+\vp)^{3}\bigg)\d x.
	\end{aligned}
	\end{equation*}
	First, from~\eqref{def:kappa} we have 
	\begin{equation*}
	E'''(\phi,\phi,\phi)=-4\alpha \|\phi\|_{L^{2}}^{2}.
	\end{equation*}
	Therefore, by change of variables,
	\begin{equation*}
	-\frac{1}{6}\int_{\RR^{N}}f''(Q)V^{3}\d x=-\frac{a^{3}}{6}E'''(\phi,\phi,\phi)=\frac{2}{3}\alpha\|\phi\|_{L^{2}}^{2}a^{3}.
	\end{equation*}
	Then, from $-\Delta \Phi+\Phi-f'(Q)\Phi=0$, $V=a\Phi$ and integration by parts,
	\begin{equation*}
	\begin{aligned}
	E_{1}&=\frac{a^{2}}{2}\int_{\RR^{N}}\left(-\Delta\Phi+\Phi-f'(Q)\Phi\right)\Phi\d x=0,\\
	E_{2}&=a\int_{\RR^{N}}\left(-\Delta\Phi+\Phi-f'(Q)\Phi\right)\vp\d x=0.
	\end{aligned}
	\end{equation*}
	By the Sobolev embedding theorem and the AM-GM inequality,
	\begin{equation*}
	\left|E_{3}\right|\lesssim \|\vec{\varphi}\|_{\E}^{2}+|\ell|^{2}+|\beta|^{2}+b^{2}\lesssim \mathcal{N}^{2}.
	\end{equation*}
	Next, from integration by parts and $-\Delta Q+Q-f(Q)=0$, we have 
	\begin{equation*}
	E_{4}=\int_{\RR^{N}}\left(-\Delta Q+Q-f(Q)\right)(V+\vp)\d x=0.
	\end{equation*}
	Last, from the Sobolev embedding Theorem and the AM-GM inequality, we have
	\begin{equation*}
	\begin{aligned}
	\left|E_{5}\right|
	&\lesssim \int_{\RR^{N}}|f''(Q)|\left(|\vp|V^{2}+|V|\vp^{2}+|\vp|^{3}\right)\d x
	\lesssim a^{4}+\mathcal{N}^{2},\\
	\left|E_{6}\right|&\lesssim \int_{\RR^{N}}\left(|V+\vp|^{4}+|V+\vp|^{p+1}\right)\d x\lesssim a^{4}+|a|^{p+1}+\mathcal{N}^{4}+\mathcal{N}^{p+1}.
	\end{aligned}
	\end{equation*} 
	Gathering above estimates, we find~\eqref{est:Eua3}.
	\end{proof}
\subsection{Time evolution analysis}\label{sse:evo}
We introduce new parameters and functionals to analyze the time evolution of solutions in the framework of Proposition~\ref{pro:dec}. It is convenient to consider together all the damped components, as follows:
\begin{align}
\mathcal{S}=|\ell|^{2}+|\beta|^{2}+\sum_{k=1}^{K}(a_{k}^{-})^{2}+b^{2} \quad \text{and} \quad
\mathcal{F}=\mathcal{E}+\mu^{-1}\mathcal{S}.
\end{align}
We also define the unstable component:
\begin{align} \mathcal{A} =\sum_{k=1}^{K}(a^{+}_{k})^{2}. \end{align}
Finally, for the analysis of the main ODE system in the degree-1 degenerate case, we will crucially relie on the following quantities
\begin{align}
\mathcal{R}_{1}&=a+\frac{b}{2\alpha} \quad \text{and}  \quad \mathcal{R}_{2}=\frac{2\alpha}{3}a^{3}+\frac{1}{2\alpha}ab^{2}+a^{2}b.
\end{align}
Indeed, $\q R_1$ and $\q R_2$ enjoy monotonicity related properties as it is made explicit in the following lemma, in which we also 
rewrite the estimates of Lemma~\ref{le:equzab}, ~\ref{le:equexp} and \ref{le:ener1}, using our new notations.

\begin{lemma}\label{le:Lyia} In the context of Proposition~\ref{pro:dec}, there exist $C_{0}>1$ such that the following hold.
	\begin{enumerate}	
		\item \emph{Coercivity and bound.} We have 
		\begin{align} \label{est:coer2}
		C_0^{-1} \q N^2 \le \q F + \mu^{-1} \q A \le C_0 \q N^2.
		\end{align}
		
		\item \emph{Liapunov evolution of $\mathcal{R}_{1}$.} We have 
		\begin{align}
		\frac{\d}{\d t}\mathcal{R}_{1}+\frac{1}{2}a^{2}&\le C_{0}\mathcal{N}^{2},\label{est:R1}\\
		\frac{\d}{\d t}\mathcal{R}_{2}+\alpha a^{4}&\le C_{0}\mathcal{N}^{3},\label{est:R2}\\
		\left|\frac{\d}{\d t}\mathcal{R}_{1}+a^{2}\right|&
		\le C_{0}\left(\mathcal{N}^{2}+|a|\mathcal{N}+|a|^{\bar{p}}\right).\label{est:R1a2}
		\end{align}
		
		\item \emph{Damped evolution}. There exists $\nu_{1}>2\mu$ such that
		\begin{equation}\label{est:dtS}
		\frac{\d}{\d t}\mathcal{S} +\nu_{1}\mathcal{S}\le C_{0}\left(\mathcal{N}^{3}+a^{2}\mathcal{N}\right),
		\end{equation}
		\begin{equation}\label{est:dtF}
		\frac{\d}{\d t}\mathcal{F} +2\mu\mathcal{F}\le C_{0}\left(\mathcal{N}^{3}+a^{2}\mathcal{N}\right).
		\end{equation}
		
		\item \emph{Exponential growth}. There exists $\nu_{2}\ge\nu_{3}>2\mu$ such that 
		\begin{align}
		\frac{\d}{\d t}\mathcal{A}- \nu_{2}\mathcal{A}&\le C_{0}\left(\mathcal{N}^{3}+a^{2}\mathcal{N}\right),\label{est:dA1}\\
		\frac{\d}{\d t}\mathcal{A}- \nu_{3}\mathcal{A}&\ge -C_{0}\left(\mathcal{N}^{3}+a^{2}\mathcal{N}\right)\label{est:dA2}.
		\end{align}

	\end{enumerate}
\end{lemma}
\begin{proof}
	Proof of (i). Due to the bound from below in \eqref{coer:E},
\[ \q F + \mu^{-1} \q A \ge \mu \| \vec \varphi \|_{\q H}^2 + \mu^{-1}( |\ell|^2 + |\beta|^2 + b^2) \gtrsim \q N^2. \]
Now $|a^\pm_k| = |\langle \vec \varphi, \vec Z_k^\pm \rangle| \le \| \vec \varphi \|_{\q H} \le \q N$, so that using the bound from above in \eqref{coer:E},
\[ \q F +  \mu^{-1} \q A \lesssim \q E + |\ell |^2 + |\beta|^2 + b^2 + \sum_{k=1}^K (a_k^-)^2 + (a_k^+)^2 \lesssim \q N^2. \]
	
	Proof of (ii). First, from~\eqref{est:a} and~\eqref{est:ba2}, we have 
	\begin{equation*}
	\begin{aligned}
	\frac{\d}{\d t}\mathcal{R}_{1}=\dot{a}+\frac{\dot{b}}{2\alpha}=-a^{2}+O(\mathcal{N}^{2}+|a|\mathcal{N}+|a|^{3}+|a|^{p}).
	\end{aligned}
	\end{equation*}
	Based on~\eqref{small}, the above estimate and the AM-GM inequality, we obtain~\eqref{est:R1} and~\eqref{est:R1a2} for $C_{0}$ large enough. Second, using again~\eqref{est:a},~\eqref{est:ba2}, we compute
	\begin{equation*}
	\begin{aligned}
	\frac{\d}{\d t}a^{3}&=3 a^{2}b+O\left(a^{2}\mathcal{N}^{2}+|a|^{3}\mathcal{N}\right),\quad  \frac{\d}{\d t}\left(ab^{2}\right)
	=-4\alpha ab^{2}+O\left(\mathcal{N}^{3}+|a|^{3}\mathcal{N}\right),\\
	\frac{\d}{\d t}(a^{2}b)&=-2\alpha a^{4}+2ab^{2}-2\alpha a^{2}b+O\left(\mathcal{N}^{4}+|a|^{3}\mathcal{N}+|a|^{5}+|a|^{p+2}\right).
	\end{aligned}
	\end{equation*}
	Combining the above estimates, we infer 
	\begin{equation*}
	\begin{aligned}
	\frac{\d}{\d t}\mathcal{R}_{2}
	&=\frac{2\alpha}{3}\frac{\d}{\d t}a^{3}+\frac{1}{2\alpha}\frac{\d}{\d t}(ab^{2})+\frac{\d}{\d t}(a^{2}b)\\
	&=-2\alpha a^{4}+O\left(\mathcal{N}^{3}+|a|^{3}\mathcal{N}+|a|^{5}+|a|^{p+2}\right),
	\end{aligned}
	\end{equation*}
	Using the AM-GM inequality (to see  $|a|^3 \q N \le \q N^{5/2} + |a|^5$, $a^2 \q N^2 \le \q N^3 + |a|^5$), and as $p+2>4$, this implies~\eqref{est:R2} for $C_{0}$ large enough.

\smallskip
	
	Proof of (iii) and (iv). The estimates~\eqref{est:dtS},~\eqref{est:dtF},~\eqref{est:dA1} and~\eqref{est:dA2} are consequences of~\eqref{equ:bl},~\eqref{est:ba2},~\eqref{equ:exp},~\eqref{dtE} and taking $0<\mu<\min_{k,\pm}(1,\alpha,\nu_{k}^{\pm})$ in Lemma~\ref{le:ener1}.
	\end{proof}

\section{Proof of Theorem~\ref{thm:1} and Theorem~\ref{thm:2}} \label{sec:3}

In this section, we prove Theorem~\ref{thm:1} and Theorem~\ref{thm:2}. First, we prove the following trapping Proposition for packed solution.
\begin{proposition}\label{prop:trapping}
There exists $\delta_0>0$ such that the following holds.
 Let  $\vec{u}=(u,\partial_{t}u)$ be a packed solution of~\eqref{equDKG}, with cluster point $(q,0)$ where $q$ is a non-degenerate state or degree-1 excited state. Due to Proposition \ref{prop:1.1}, for any $0 < \delta < \delta_0$, there exists $T_{\delta}>0$ and $\xi_{\delta}\in \RR^{N}$ such that 
\begin{equation}\label{est:ini}
\|u(T_{\delta})-q(\cdot-\xi_{\delta})\|_{H^{1}}+\|\partial_{t}u(t)\|_{L^{2}}+\bigg[\int_{T_{\delta}}^{\infty}\|\partial_{t}u(t)\|_{L^{2}}^{2}\d t\bigg]^{\frac{1}{2}}<\delta.
\end{equation}
Then $\vec{u}$ admits a decomposition as in Proposition~\ref{pro:dec} for all $t\in [T_{\delta},\infty)$ and satisfies
\begin{equation}\label{est:allt}
\forall t \ge T_\delta, \quad \|u(t)-q(\cdot-\xi_{\delta})\|_{H^{1}}+\|\partial_{t} u(t)\|_{L^{2}}\lesssim \delta^{\frac{1}{2}}.
\end{equation}
\end{proposition}

\begin{proof} 
For $\delta < \gamma$ small enough, the existence of $T_{\delta}$ is a direct consequence of Proposition \ref{prop:1.1}: $\vec{u}(T_{\delta})$ admits a decomposition as in Proposition~\ref{pro:dec} and satisfies
	\begin{equation}\label{est:Tdelta}
	 |z(T_{\delta})-\xi_{\delta}|+|\theta(T_{\delta})|+|a(T_{\delta})|+\mathcal{N}(T_{\delta})\lesssim \delta.
	\end{equation}
	
	For a constant $C>1$ to be chosen later, we introduce the following bootstrap estimate
	\begin{equation}\label{Boot}
	|z-\xi_{\delta}|+|\theta|\le \delta^{\frac{1}{2}},\quad |a|\le C\delta^{\frac{2}{3}},\quad \mathcal{N}\le C\delta,\quad  \mathcal{A}\le 2C\delta^{2}.
	\end{equation}
	Set 
	\begin{equation*}
		T^{*}=\sup\left\{t\in[T_{\delta},\infty)\ \text{such that}~\eqref{est:dec}\ \text{and}\ ~\eqref{Boot}~\text{holds on}\ [T_{\delta},t]\right\}.
	\end{equation*}
	We prove thqat $T^{*}=\infty$ by strictly improving the bootstrap assumption~\eqref{Boot} on $[T_{\delta},T^{*})$ (upon chosing $C$ large enough). In the remainder of the proof, the implied constants in $\lesssim$ or $O$ do not depend on $\delta$ nor on the constant $C$  appearing in the bootstrap assumption~\eqref{Boot}. Recall that, in the case when $q$ is a non-degenerate state, we denote $a=b=0$.
	
	\smallskip
	
\emph{Step 1. Preliminary bounds.} 
Let $t \in [T_\delta, T_*)$. Integrating \eqref{est:dA2} on $[T_{\delta},t]$, and using~\eqref{est:Tdelta}, we have 
		\begin{align*}
	\int_{T_{\delta}}^{t}\mathcal{A}(s)\d s
	&\lesssim \mathcal{A}(t)
	+\int_{T_{\delta}}^{t}\mathcal{N}^{3}(s)\d s+\int_{T_{\delta}}^{t}|a|^{3}(s)\d s\\
	&\lesssim C\delta^{2}+C\delta\int^{t}_{T_{\delta}}\mathcal{N}^{2}(s)\d s
	+C\delta^{\frac{2}{3}}\int_{T_{\delta}}^{t}a^{2}(s)\d s,
	\end{align*}
Similarly, integrating\eqref{est:dtF}, we get
		\begin{align*}
	\int_{T_{\delta}}^{t}\mathcal{F}(s)\d s
	&\lesssim \left|\mathcal{F}(t)\right|+\left|\mathcal{F}(T_{\delta})\right|
	+\int_{T_{\delta}}^{t}\mathcal{N}^{3}(s)\d s+\int_{T_{\delta}}^{t}|a|^{3}(s)\d s\\
	&\lesssim C^{2}\delta^{2}+C\delta\int^{t}_{T_{\delta}}\mathcal{N}^{2}(s)\d s
	+C\delta^{\frac{2}{3}}\int_{T_{\delta}}^{t}a^{2}(s)\d s,
	\end{align*}
	and integrating~\eqref{est:R1}, there hold
    	\begin{equation*}
    \frac{1}{2}\int_{T_{\delta}}^{t}a^{2}(s)\d s-C_{0}\int_{T_{\delta}}^{t}\mathcal{N}^{2}(s)\d s
     \le R_{1}(T_{\delta})-R_{1}(t)\lesssim C\delta^{\frac{2}{3}}+C\delta.
    \end{equation*}
 Now recall that due to the coercivity \eqref{est:coer2}, $\q F + \mu^{-1}\q A \ge C_0^{-1} \q N^2$, hence
\begin{align*}
\MoveEqLeft\int_{T_{\delta}}^{t}\left(\mathcal{N}^{2}(s)+a^{2}(s)\right)\d s \\
& \le \int_{T_{\delta}}^{t}a^{2}(s)\d s-C_{0}\int_{T_{\delta}}^{t}\mathcal{N}^{2}(s)\d s + C_0(C_0+1) \int_{T_{\delta}}^{t}(\q F(s)+\mu^{-1}\mathcal{A}(s)) \d s \\
& \lesssim C\delta^{\frac{2}{3}}+C\delta+C^{2}\delta^{2} +C\delta\int^{t}_{T_{\delta}}\mathcal{N}^{2}(s)\d s
	+C\delta^{\frac{2}{3}}\int_{T_{\delta}}^{t}a^{2}(s)\d s.
\end{align*}
Taking $0<\delta\ll 1$ small enough, we infer
\begin{align*}
\int_{T_{\delta}}^{t}\left(\mathcal{N}^{2}(s)+a^{2}(s)\right)\d s \lesssim  C\delta^{\frac{2}{3}}+C\delta+C^{2}\delta^{2}\lesssim C\delta^{\frac{2}{3}}+C^{2}\delta^{2}.
\end{align*}

\emph{Step 2. Estimate on $z$ and $\theta$.} By~\eqref{equ:zta} and~\eqref{equ:bl}, we have 
\begin{align*}
\MoveEqLeft |z(t)-\xi_{\delta}|+|\theta(t)| \lesssim \left|z(t)-\xi_{\delta}+\frac{\ell(t)}{2\alpha}\right|+\left|\theta(t)+\frac{\ell(t)}{2\alpha}\right|+|\ell(t)|+|\beta(t)|\\
	&\lesssim \mathcal{N}(t)+\mathcal{N}(T_{\delta})+|z(T_{\delta})-\xi_{\delta}|+|\theta(T_{\delta})|+\int_{T_{\delta}}^{t}\left(\mathcal{N}^{2}(s)+a^{2}(s)\right)\d s\\
	& \lesssim	 C\delta+ C\delta^{\frac{2}{3}}+C^{2}\delta^{2},
\end{align*}	
which strictly improves the estimate~\eqref{Boot} of $z$ and $\theta$ on $[T_{\delta},T^{*})$ for $\delta$ small enough.
	
\emph{Step 3. Estimate on $a$.} Due to the energy dissipation \eqref{eq:energy_packed} and the initial assumption \eqref{est:ini}, note that we have 
	\begin{equation*}
	E(q,0)\le E(\vec{u}(t))\le E(q,0)+O(\delta^{2}).
	\end{equation*}
	Using the expansion of the energy~\eqref{est:Eua3} (and the bootstrap bounds~\eqref{Boot}), we deduce that 
	\begin{equation*}
	\begin{aligned}
	|a|^{3} & \lesssim \delta^{2}+\mathcal{N}^{2}+a^{4}+|a|^{p+1} \\
	& \lesssim C^{2}\delta^{2}+C^{\bar{p}+1}\delta^{\frac{2}{3}(\bar{p}+1)}.
	\end{aligned}
	\end{equation*}
	It follows that
	\begin{equation*}
	|a| \lesssim \delta^{\frac{2}{3}}+ C^{\frac{2}{3}}\delta^{\frac{2}{3}}+C^{\frac{\bar{p}+1}{3}}\delta^{\frac{2}{9}(\bar{p}+1)}.
	\end{equation*}
	which strictly improves the estimate~\eqref{Boot} of $a$ on $[T_{\delta},T^{*})$ for $C$ large enough and $\delta$ small enough.
	
\emph{Step 4. Estimate on $\mathcal{N}$}. Under the bootstrap assumption \eqref{Boot}, estimate ~\eqref{est:dtF} writes 
\begin{equation*}
\begin{aligned}
	\frac{\d}{\d t}\mathcal{F}+2\mu\mathcal{F}
&\lesssim C^{3}\delta^{\frac{7}{3}}+C^{3}\delta^{3}\lesssim C^{3}\delta^{\frac{7}{3}}.
		\end{aligned}
	\end{equation*}
Let $t \in [T_{\delta},T^{*})$. Integrating the above estimate on $[T_{\delta},t]$, using also the initial time assumption~\eqref{est:Tdelta}, yields
	\begin{equation}\label{est:F}
	\begin{aligned}
	\mathcal{F}(t)&\lesssim e^{-2\mu (t-T_{\delta})}\mathcal{F}(T_{\delta})+C^{3}\int_{T_{\delta}}^{t}e^{-2\mu(t-s)}
	\delta^{\frac{7}{3}}\d s\lesssim \delta^{2}+C^{3}\delta^{\frac{7}{3}}.
	\end{aligned}
	\end{equation}
We now use once more the coercivity \eqref{est:coer2}, the~\eqref{Boot} on $\q A$ so that the estimate \eqref{est:F} above gives, for $C$ large enough
	\begin{equation*}
	\begin{aligned}
		\mathcal{N}^{2}(t)
		&\lesssim \mathcal{F}+\mathcal{A}\lesssim \delta^{2}+C^{3}\delta^{\frac{7}{3}}+C\delta^{2}
		\le 2C\delta^{2}.
	\end{aligned}
	\end{equation*}
	This strictly improves the estimate~\eqref{Boot} of $\mathcal{N}$ on $[T_{\delta},T^{*})$.
	
	\emph{Step 5. Estimate on $\mathcal{A}$}. Let us first rewrite the estimates~\eqref{est:dtS},~\eqref{est:dA1},~\eqref{est:dA2} in the context of the bootstrap assumption~\eqref{Boot}: for all $t\in [T_{\delta},T_{*})$, 
	\begin{align}
	\frac{\d }{\d t}\mathcal{S}+\nu_{1}\mathcal{S}&\le C_{0} \left(C^{3}\delta^{\frac{7}{3}}+C^{3}\delta^{3}\right)\le C_{0}C^{3}\delta^{\frac{7}{3}},\label{est:dtSS}\\
	\frac{\d }{\d t}\mathcal{A}-\nu_{2}\mathcal{A}&\le C_{0} \left(C^{3}\delta^{\frac{7}{3}}+C^{3}\delta^{3}\right)\le C_{0}C^{3}\delta^{\frac{7}{3}},\label{est:dtA1}\\
	\frac{\d }{\d t}\mathcal{A}-\nu_{3}\mathcal{A}&\ge -C_{0}\left(C^{3}\delta^{\frac{7}{3}}+C^{3}\delta^{3}\right)\ge  -C_{0}C^{3}\delta^{\frac{7}{3}}.\label{est:dtA2}
	\end{align} 
	Then for $t\in [T_{\delta},T^{*})$, we have (integrating~\eqref{est:dtSS} on $[T_{\delta},t]$)
	\begin{equation}\label{est:S}
	\q S(t)\lesssim e^{-\nu_{1}(t-T_{\delta})}S(T_{\delta})+C^{3}\int_{T_{\delta}}^{t}e^{-2\mu(t-s)}
	\delta^{\frac{7}{3}}\d s\lesssim \delta^{2}+C^{3}\delta^{\frac{7}{3}}.
	\end{equation}
	We now prove by contradiction that for $C$ large enough, it holds
	\begin{equation}\label{est:A}
	\forall t\in [T_{\delta},T_{*}),\quad \mathcal{A}(t)< 2C\delta^{2}.
	\end{equation}
	For the sake of contradiction, assume that there exists $t_{2}\in[T_{\delta},T^{*})$ such that 
	\begin{equation*}
		\mathcal{A}(t_{2})=2C\delta^{2}\quad \text{and for all } t \in [T_{\delta},t_{2}), \quad 
		\mathcal{A}(t)<2C\delta^{2}.
	\end{equation*}
	On the one hand, by continuity of $\mathcal{A}$, there exists $t_{1}\in[T_{\delta},t_{2}]$ such that 
	\begin{equation}\label{est:A1A2}
		\mathcal{A}(t_{1})=C\delta^{2}\quad \text{and for all } t \in (t_{1},t_{2}), \quad 
		C\delta^{2}<\mathcal{A}(t)<2C\delta^{2}.
	\end{equation}
	For $t \in [t_1,t_2]$, divide \eqref{est:dtA1} and \eqref{est:dtA2} by $\q A$ (using~\eqref{est:A1A2}), and then integrate on $[t_1,t_2]$: it follows that 
	\begin{equation}\label{est:t2-t1}
	\frac{\log 2}{\nu_{2}}+O(C^{2}\delta^{\frac{1}{3}})\le t_{2}-t_{1}\le 
	\frac{\log 2}{\nu_{3}}+O(C^{2}\delta^{\frac{1}{3}}).
	\end{equation}
	Therefore, using \eqref{est:A1A2} again
	\begin{equation}\label{est:inta1}
	\int_{t_{1}}^{t_{2}}\mathcal{A}(t)\d t\ge C\delta^{2}(t_{2}-t_{1})\ge C\delta^{2}\left(\frac{\log 2}{\nu_{2}}+O(C^{2}\delta^{\frac{1}{3}})\right).
	\end{equation}
	On the other hand, by the definition of $\varphi_2$ in ~\eqref{def:vp} and the bound on $\q S$ ~\eqref{est:S}, for any $t\in[t_{1},t_{2}]$
	\begin{equation*}
		\|\vpp(t)\|_{L^{2}}^{2}\lesssim \|\partial_{t}u(t)\|^{2}_{L^{2}}+\mathcal{S}(t)\lesssim 
		\|\partial_{t}u(t)\|^{2}_{L^{2}}+\delta^{2}+C^{3}\delta^{\frac{7}{3}}.
	\end{equation*}
	Thus, from~\eqref{est:ini} and~\eqref{est:t2-t1},
	\begin{equation}\label{est:vpp}
	\int_{t_{1}}^{t_{2}}\|\vpp(t)\|_{L^{2}}^{2}\d t\lesssim \int_{t_{1}}^{t_{2}}\left(\|\partial_{t}u(t)\|^{2}_{L^{2}}+\delta^{2}+C^{3}\delta^{\frac{7}{3}}\right)\d t
	\lesssim \delta^{2}+C^{3}\delta^{\frac{7}{3}}.
	\end{equation}
	By the definition of $a_{k}^{\pm}$, we have
	\begin{equation*}
		a_{k}^{+}=\zeta_{k}^{+}\langle \vp, \Upsilon_{k}\rangle +\langle \vpp,\Upsilon_{k}\rangle,\quad 
		a_{k}^{-}=\zeta_{k}^{-}\langle \vp, \Upsilon_{k}\rangle +\langle \vpp,\Upsilon_{k}\rangle
	\end{equation*}
	and thus for all $k=1,\cdots,K$
	\begin{equation*}
		a_{k}^{+}=\frac{\zeta^{+}_{k}}{\zeta^{-}_{k}}a_{k}^{-}+\frac{\zeta^{-}_{k}-\zeta^{+}_{k}}{\zeta^{-}_{k}}\langle \vpp,\Upsilon_{k}\rangle,
	\end{equation*}
	from where we see that $\q A \lesssim \q S + \| \varphi_2 \|_{L^2}^2$. Gathering estimates~\eqref{est:S} and~\eqref{est:vpp}, we find
	\begin{equation*}
		\int_{t_{1}}^{t_{2}}\mathcal{A}(t)\d t\lesssim \int_{t_{1}}^{t_{2}}\mathcal{S}(t)\d t+\int_{t_{1}}^{t_{2}}\|\vpp(t)\|_{L^{2}}^{2}\d t\lesssim \delta^{2}+C^{3}\delta^{\frac{7}{3}},
	\end{equation*}
	which is a contradiction with~\eqref{est:inta1} for $C$ large enough and $\delta$ small enough. This proves \eqref{est:A}, and this strictly improves the estimate on $\q A$
	
	\emph{Step 6. Conclusion}. As a consequence of improving the bootstrap assumption \eqref{Boot} on $z -\xi_\delta$, $\theta$, $\mathcal{N}$ and $\mathcal{A}$, we conclude that $T^{*}=\infty$. 
	
	Finally, from~\eqref{Boot}, we know that, for all $t\in[T_{\delta},+\infty)$,
	\begin{equation*}
	\begin{aligned}
	&\|u(t)-q(\cdot-\xi_{\delta})\|_{H^{1}}+\|\partial_{t}u(t)\|_{L^{2}}\\
	&\lesssim \|u(t)-Q(t)\|_{H^{1}}+\|Q(t)-q(\cdot-\xi_{\delta})\|_{H^{1}}+\|\partial_{t}u(t)\|_{L^{2}}\\
	&\lesssim \mathcal{N}(t)+|a(t)|+|z(t)-\xi_{\delta}|+|\theta(t)|
	\lesssim \delta^{\frac{1}{2}}+C\delta^{\frac{2}{3}}+C\delta,
	\end{aligned}
	\end{equation*}
	which implies~\eqref{est:allt} for $\delta$ small enough. The proof of Proposition~\ref{prop:trapping} is complete.
\end{proof}

\begin{remark}
Observe that the proof of Proposition \ref{prop:trapping} actually proves that 
\[ \theta(t), a(t), \q N(t), \q A(t) \to 0 \quad \text{as} \quad t \to +\infty, \]
a fact that we will use in the proofs of Theorem \ref{thm:1} and \ref{thm:2}.
\end{remark}

\begin{proof}[End of the proof of Theorem~\ref{thm:1}] We assume here that  $q$ is a  non-degenerate bound state, and prove exponential convergence. Recall that, in this context, we have set $a=b=0$.
	
	First, from Proposition~\ref{prop:trapping}, we know that, for any $0<\delta<\delta_{0}^{2}$, there exists $T^{*}_{\delta}\gg 1$ and $\xi_{\delta}\in \RR^{N}$ such that, $\vec{u}$ admits a decomposition as in Proposition~\ref{pro:dec} for all $t\in [T_{\delta}^{*},\infty)$ and satisfies
	\begin{equation}\label{est:Td}
	\|u(t)-q(\cdot-\xi_{\delta})\|_{H^{1}}+\| \partial_{t}u(t)\|_{L^{2}}+|z(t)-\xi_{\delta}|+|\theta(t)|+\mathcal{N}(t)<\delta.
	\end{equation}

\emph{Step 1.} We claim that
	\begin{equation}\label{est:dtFA}
	\forall t \ge T_\delta^*, \quad \frac{\d}{\d t}\left(\mathcal{F}(t)+\mu^{-1}\mathcal{A}(t)\right)\le -\mu \left(\mathcal{F}(t)+\mu^{-1}\mathcal{A}(t)\right).
	\end{equation}
	This will implies exponential decay of $\mathcal{F}+\mu^{-1}\mathcal{A}$. So as to prove the above differential inequality, it is convenient to introduce the auxiliary quantity $\tilde{\mathcal{A}}:=\mathcal{A}-\delta^{\frac{1}{2}}\mathcal{F}$. 
	First note that, from~\eqref{est:coer2} and~\eqref{est:Td},
	\begin{equation}\label{est:N3}
	\mathcal{N}^{3}(t)\le \delta \mathcal{N}^{2}(t)\le \delta C_{0} \left(\mathcal{F}+\mu^{-1}\mathcal{A}\right).
	\end{equation}
	As a consequence,
	\begin{equation*}
	\tilde{\mathcal{A}}=(1+\delta^{\frac{1}{2}}\mu^{-1})\mathcal{A}-\delta^{\frac{1}{2}}\left(\mathcal{F}+\mu^{-1}\mathcal{A}\right)\le (1+\delta^{\frac{1}{2}}\mu^{-1})\mathcal{A}.
	\end{equation*}
Therefore, using the evolution equations ~\eqref{est:dtF} and~\eqref{est:dA2} for $\q F$ and $\q A$,  and in view of~\eqref{est:N3},
\begin{equation*}
\begin{aligned}
\frac{\d}{\d t}\tilde{\mathcal{A}}
&=\frac{\d}{\d t}\mathcal{A}
-\delta^{\frac{1}{2}}\frac{\d}{\d t}\mathcal{F} \ge \left(\nu_{3}-2\delta^{\frac{1}{2}}\right)\mathcal{A}+2\mu\delta^{\frac{1}{2}}\left(\mathcal{F}+\mu^{-1}\mathcal{A}\right)-(1+\delta^{\frac{1}{2}})C_{0}\mathcal{N}^{3}\\
&\ge \left(\nu_{3}-2\delta^{\frac{1}{2}}\right)\mathcal{A}+
\left(2\mu\delta^{\frac{1}{2}}-\delta C_{0}^{2}(1+\delta^{\frac{1}{2}})\right)\left(\mathcal{F}+\mu^{-1}\mathcal{A}\right)\ge 
\frac{1}{2}\nu_{3}\tilde{\mathcal{A}},
\end{aligned}
\end{equation*}
where $\delta>0$ is so small that
\begin{equation*}
\left(\nu_{3}-2\delta^{\frac{1}{2}}\right)\ge \frac{1}{2}\nu_{3}(1+\delta^{\frac{1}{2}}\mu^{-1}),\quad 
2\mu\delta^{\frac{1}{2}}-\delta\mu^{-1} C_{0}(1+\delta^{\frac{1}{2}})\ge 0.
\end{equation*}
Integrating on $[t,s]\subseteq [T^{*}_{\delta},\infty)$, we get
\begin{equation*}
\tilde{\q A}(t)\le e^{\frac{1}{2}\nu_{3}(t-s)}\tilde{\q A}(s).
\end{equation*}
We now take  the limit $s\to\infty$ and using that $\tilde {\q A}(s)$ is bounded due to~\eqref{est:Td}, we obtain that $\tilde {\q A}(t) \le 0$, that is
\begin{equation}\label{est:A-F}
\forall t\in [T^{*}_{\delta},\infty),\quad \mathcal{A}(t)\le \delta^{\frac{1}{2}}\mathcal{F}(t)\le \delta^{\frac{1}{2}}\left(\mathcal{F}(t)+\mu^{-1}\mathcal{A}(t)\right).
\end{equation}
Combining~\eqref{est:dtF},~\eqref{est:dA1},~\eqref{est:N3} and~\eqref{est:A-F}, we have 
\begin{equation*}
\begin{aligned}
\frac{\d}{\d t}\left(\mathcal{F}+\mu^{-1}\mathcal{A}\right)&=\frac{\d}{\d t}\mathcal{F}+\mu^{-1}\frac{\d}{\d t}\mathcal{A}\\
&\le -2\mu(\mathcal{F}+\mu^{-1}\mathcal{A})+(\mu^{-1}\nu_{2}+2)\mathcal{A}+C_{0}(1+\mu^{-1})\mathcal{N}^{3}\\
&\le(-2\mu +\delta^{\frac{1}{2}}(\mu^{-1}\nu_{2}+2)+\delta C^{2}_{0}(1+\mu^{-1})) \left(\mathcal{F}+\mu^{-1}\mathcal{A}\right),
\end{aligned}
\end{equation*}
which implies~\eqref{est:dtFA} for $\delta$ small enough.

\emph{Step 2.} Integrating~\eqref{est:dtFA} on $[T^{*}_{\delta},t]$ and using~\eqref{est:Td}, we have
\begin{equation*}
\left(\mathcal{F}+\mu^{-1}\mathcal{A}\right)(t)\le e^{-\mu(t-T^{*}_{\delta})}\left(\mathcal{F}+\mu^{-1}\mathcal{A}\right)\left(T^{*}_{\delta}\right)\lesssim \delta^{2}e^{-\mu (t-T^{*}_{\delta})}.
\end{equation*}
Recall once again \eqref{est:coer2}: it follows that 
\begin{equation}\label{est:N}
\mathcal{N}^{2}(t)\lesssim (\mathcal{F}+\mu^{-1}\mathcal{A})(t)\lesssim e^{-\mu t}.
\end{equation}
 Now we can integrate \eqref{equ:zta}, using \eqref{est:N} and recalling that $a=0$: this proves that there exists $z_{\infty}\in \RR^{N}$ such that $z(t) \to z_\infty$ as $t \to +\infty$, and
\begin{equation}\label{est:zta}
\forall t\ge T_{\delta},\quad |z(t)-z_{\infty}|+|\theta(t)|\lesssim \int_{t}^{\infty}\mathcal{N}(s)^2 \d s\lesssim e^{-\mu t}.
\end{equation}
Gathering estimates~\eqref{est:N} and~\eqref{est:zta}, we obtain
\begin{align*}
\MoveEqLeft \|u(t)-q(\cdot-z_{\infty})\|_{H^{1}}+\|\partial_{t} u(t)\|_{L^{2}} \\
&\lesssim \|u(t)-Q(t)\|_{H^{1}}+\|Q(t)-q(\cdot-z_{\infty})\|_{H^{1}}+\|\partial_{t} u(t)\|_{L^{2}}\\
&\lesssim \mathcal{N}(t)+|z(t)-z_{\infty}|+|\theta(t)|\lesssim e^{-\mu t}.
\end{align*}

The proof of Theorem~\ref{thm:1} is complete.
	\end{proof}

\begin{proof}[End of the proof of Theorem~\ref{thm:2}]
	We here assume that $q$ is a degree-1 excited state, and prove algebraic convergence of $\vec u$ to $(q,0)$.
	
	As earlier, we know from Proposition~\ref{prop:trapping}, that for any $0<\delta<\delta_{0}^{2}$, there exists $T^{*}_{\delta}\gg 1$ (we can choose $T_\delta^* = T_{\delta^2/2}$) and $\xi_{\delta}\in \RR^{N}$ such that, $\vec{u}$ admits a decomposition as in Proposition~\ref{pro:dec} for all $t\in [T_{\delta}^{*},\infty)$ and satisfies
	\begin{equation}\label{est:Tdexci}
	\|u(t)-q(\cdot-\xi_{\delta})\|_{H^{1}}+\| \partial_{t}u(t)\|_{L^{2}}+|z(t)-\xi_{\delta}|+|a(t)|+|\theta(t)|+\mathcal{N}(t)<\delta.
	\end{equation}
	
	\emph{Step 1.} We claim that there exists $C_1 \ge 1$ (independent of $\delta$) such that 
	\begin{equation}\label{est:AFa}
	\forall t \ge T_{\delta}^{*}, \quad \mathcal{A}(t) \le C_1 \delta^{\frac{1}{4}} (a(t) ^{2}+ \mathcal{N}(t)^{2}).
	\end{equation}
	Again, we consider an auxiliary quantity, which is here $\tilde{\mathcal{A}}:=\mathcal{A}-\delta^{\frac{1}{4}}\mathcal{F}-\delta^{-\frac{1}{2}}\mathcal{R}_{2}$. From the evolution equations~\eqref{est:R2},~\eqref{est:dtF} and~\eqref{est:dA2} of $\q A$, $\q F$ and $\q R_2$, we bound
	\begin{align*}
	\frac{\d}{\d t}\tilde{\mathcal{A}}
	& = \frac{\d}{\d t}\mathcal{A}-\delta^{\frac{1}{4}}\frac{\d}{\d t}\mathcal{F}-\delta^{-\frac{1}{2}}\frac{\d}{\d t}\mathcal{R}_{2}\\
	& \ge 2\mu \delta^{\frac{1}{4}}\left(\mathcal{F}+\mu^{-1}\mathcal{A}\right)+(\nu_{3}-2\delta^{\frac{1}{4}})\mathcal{A}+\delta^{-\frac{1}{2}}\alpha a^{4}\\
	&\qquad -C_{0}(1+\delta^{-\frac{1}{2}}+\delta^{\frac{1}{4}})\mathcal{N}^{3} -C_{0}\left(1+\delta^{\frac{1}{4}}\right)a^{2}\mathcal{N}.\\
	& \ge 2\mu^{2}\delta^{\frac{1}{4}}\mathcal{N}^{2} +(\nu_{3}-2\delta^{\frac{1}{4}})\mathcal{A}+\delta^{-\frac{1}{2}}\alpha a^{4}
	-3C_{0}\delta^{-\frac{1}{2}}\mathcal{N}^{3}-2C_{0}a^{2}\mathcal{N}.
	\end{align*}
	Notice that from the AM-GM inequality, we have 
	\begin{equation}\label{est:a2N}
	2a^{2}\mathcal{N}=2(a^{2}\delta^{-\frac{3}{16}})(\mathcal{N}\delta^{\frac{3}{16}})\le a^{4}\delta^{-\frac{3}{8}}+\mathcal{N}^{2}\delta^{\frac{3}{8}},
	\end{equation}
	so that, rearranging the terms in the preceding inequality, we get
	\begin{equation*}
	\frac{\d}{\d t}\tilde{\mathcal{A}}\ge 
	\left(2\mu^{2}\delta^{\frac{1}{4}}-3C_{0}\delta^{\frac{1}{2}}-C_{0}\delta^{\frac{3}{8}}\right)
	\mathcal{N}^{2}+(\nu_{2}-2\delta^{\frac{1}{4}})\mathcal{A}+(\delta^{-\frac{1}{2}}\alpha-\delta^{-\frac{3}{8}}C_{0})a^{4}\ge 0,
	\end{equation*}
	where $\delta$ is so small that
	\begin{equation*}
	2\mu^{2}\delta^{\frac{1}{4}}-3C_{0}\delta^{\frac{1}{2}}-C_{0}\delta^{\frac{3}{8}}>0,\quad 
	\nu_{2}-2\delta^{\frac{1}{4}}>0\quad \text{and}\quad 
	\delta^{-\frac{1}{2}}\alpha-\delta^{-\frac{3}{8}}C_{0}>0.
	\end{equation*}
	Now, from Proposition~\ref{prop:trapping}, we know that $\tilde{\mathcal{A}}(t)\to 0$ as $t\to \infty$. As we just showed that $\tilde{\mathcal{A}}$ is non decreasing on $[T_{\delta}^{*},\infty)$, we hence conclude that for all $t \ge T_\delta^*$, $\tilde{\mathcal{A}}(t)\le 0$, or equivalently (due to the definition of $\tilde{\q A}$), that for all $t \ge T_\delta^*$,
	\begin{align*}
 \mathcal{A}(t)&\le \delta^{\frac{1}{4}}\mathcal{F}(t)+\delta^{-\frac{1}{2}}\mathcal{R}_{2}(t)\\
&\le \delta^{\frac{1}{4}}\mathcal{F}(t)+\delta^{\frac{1}{2}}\left(\frac{2\alpha}{3}+1\right)a^{2}+\frac{\delta^{\frac{1}{2}}}{2\alpha}\left(a^{2}+\mathcal{N}^{2}\right),
\end{align*}
As $\q F \lesssim \q N$, this implies~\eqref{est:AFa} for $\delta$ small enough.
	
	\smallskip
	
	\emph{Step 2.} We claim that there exists $C_{2}\ge 1$ (independent of $\delta$) such that,
	\begin{equation}\label{est:FAR1}
	\forall t \ge T_\delta^*, \quad |a(t)|\le C_{2}\left( (t-T_\delta^*+\delta^{-1})^{-1} +\mathcal{N}(t)\right).
	\end{equation}
	We use yet another auxiliary quantity, namely $\widetilde{\mathcal{R}}_{1}:=\mathcal{R}_{1}+2C_{0}^2\mu^{-1}\left(\mathcal{F}+\mu^{-1}\mathcal{A}\right)$.
	Note that, from the AM-GM inequality and~\eqref{est:Tdexci}, there exists $C_{2}\ge 1$ such that 
	\begin{equation}\label{est:R12}
	\widetilde{\mathcal{R}}_{1}^{2}\le 2\mathcal{R}_{1}^{2}+O\left(\mathcal{N}^{4}\right)\le C_{2}(a^{2}+\mathcal{N}^{2}).
	\end{equation}
	Furthermore, using the evolution equation \eqref{est:dtF} and \eqref{est:dA1} for $\q F$ and $\q A$, we have

	\begin{align*}
	\frac{\d}{\d t}\left(\mathcal{F}+\mu^{-1}\mathcal{A}\right) &\le -2\mu\left(\mathcal{F}+\mu^{-1}\mathcal{A}\right)+(\nu_{3}+2)\mathcal{A}
	+C_{0}\left(1+\mu^{-1}\right)\left(\mathcal{N}^{3}+a^{2}\mathcal{N}\right)\\
	 &\le -\mu \left(\mathcal{F}+\mu^{-1}\mathcal{A}\right)+ O(\delta^{\frac{1}{4}} (a^{2}+\mathcal{N}^{2}))
	\end{align*}
	where we also used \eqref{est:Tdexci} and \eqref{est:AFa}. Therefore, using the evolution equation \eqref{est:R1} for $\q R_1$, we deduce that,
	\begin{equation*}
	\begin{aligned}
	\frac{\d}{\d t}\widetilde{\q R}_{1}
	&=\frac{\d}{\d t}\mathcal{R}_{1}+2C_{0}^2\mu^{-1}\frac{\d}{\d t}\left(\mathcal{F}+\mu^{-1}\mathcal{A}\right)\\
	&\le -\frac{1}{2}a^{2}+C_{0}\mathcal{N}^{2}+2C_{0}^2\mu^{-1}\left(-\mu (\mathcal{F}+\mu^{-1}\mathcal{A})\right)+ O(\delta^{\frac{1}{4}}(a^{2}+\mathcal{N}^{2}))\\
	&\le -\frac{1}{2}a^{2}-C_{0}\mathcal{N}^{2}+O(\delta^{\frac{1}{4}}(a^{2}+\mathcal{N}^{2})).
	\end{aligned}
	\end{equation*}
	(We used the coercivity estimate \eqref{est:coer2}:  $\q F + \mu^{-1}\q A \ge C_0^{-1} \q N$, and also that $|b| \lesssim \q N$). For $\delta$ small enough, we infer that
	\[ \frac{\d}{\d t}\widetilde{\q R}_{1} \le -\frac{1}{4}a^{2}- \frac{C_{0}}{2} \mathcal{N}^{2} \le - c_2 \tilde {\q R}_1^2. \]
	for some universal constant $c_2>0$. Recall now that  from Proposition~\ref{prop:trapping}, $\tilde {\q R}_1(t) \to 0$ as $t \to +\infty$, and therefore, after dividing by $\tilde{\q R}_1^2 $, and integrating on $[T_\delta^*,t]$, we obtain
	\[ \forall t \ge T_\delta^*, \quad  0 \le \tilde {\q R}_1(t)\lesssim (t-T_\delta^*+\delta^{-1})^{-1}. \]
(Notice that $\tilde{\q R}_1(T_\delta^*) \lesssim \delta$, in view of \eqref{est:Tdexci}). It can be rewritten as
	\begin{equation*}
	-2C_{0}\mu^{-2}\left(\mathcal{F}(t)+\mu^{-1}\mathcal{A}(t)\right)\le a(t)+\frac{b(t)}{2\alpha}
	\lesssim  (t-T_\delta^*+\delta^{-1})^{-1}.
	\end{equation*}
	As $|b|, \mathcal{F}+\mu^{-1}\mathcal{A} \lesssim \q N$, this implies~\eqref{est:FAR1} for $C_{2}$ large enough.
	
	\bigskip
	
	\emph{Step 3. Conclusion.} Now we prove the algebraic decay rate by a bootstrap argument. For $C_3 \ge 1$ to be chosen later, we introduce the following bootstrap bounds
	\begin{equation}\label{est:Boot}
	|a(t)| \le \frac{C_{3}^{2}}{(t-T_\delta^*+\delta^{-1})} ,\quad \mathcal{N}(t)\le \frac{C_{3}}{(t-T_\delta^*+\delta^{-1})}. 
	\end{equation}
	Let 
	\begin{equation*}
	T^{**}=\sup\left\{ t\in [T_{\delta}^*,\infty)\  \text{such that}~\eqref{est:Boot}\ \text{holds on}\ [T_{\delta}^*,t]\right\}.
	\end{equation*}
	We will prove that $T^{**} =+\infty$ by strictly improving the bootstrap assumption~\eqref{est:Boot} on $[T_{\delta}^*,T^{**})$, upon choosing $\delta$ small enough. In this bootstrap, the implied constants do not depend on $\delta$ and $C_{3}$, but can depend on $C_0,C_1,C_2$.
	
	As $|a(T_\delta^*)| \le \delta$ and $\mathcal{N}(T_\delta^*) \le \delta$ due to \eqref{est:Tdexci}, the bootstrap estimate~\eqref{est:Boot} holds (strictly) at $t = T_\delta^*$, and so $T^{**} >T_\delta^*$.
	
	\emph{Estimate on $a$.} From~\eqref{est:FAR1} and the boostrap estimate~\eqref{est:Boot} on $\q N$, we have
	\begin{equation*}
	|a(t)|\le C_2 \left( \frac{1}{t-T_\delta^*+\delta^{-1}} + \q N(t) \right) \le \frac{C_2(C_{3}+1)}{(t-T_\delta^*+\delta^{-1})},
	\end{equation*}
	which strictly improves the estimate~\eqref{est:Boot} on $a$ for taking $C_{3}$ large enough (depended on $C_{2}$).
	
	\emph{Estimate on $\mathcal{N}$.} First, we claim that
	\begin{equation}\label{est:FT0}
	\forall t \in  [T_{\delta}^{*},T^{**}), \quad \mathcal{F}(t) \lesssim \frac{1}{(t-T_\delta^*+\delta^{-1})^2}.
	\end{equation}
	Indeed, using~\eqref{est:Tdexci} and the bootstrap bound~\eqref{est:Boot}, the evolution equation ~\eqref{est:dtF} on $\q F$ writes
	\begin{equation*}
	\forall t \in [T_\delta^{*},T^{**}), \quad \frac{\d}{\d t}\mathcal{F}+2\mu \mathcal{F}\le C_{0} \q N (a^2 + \q N^2) \lesssim  \frac{\delta C_{3}^{2}(C_{3}^{2}+1)}{(t-T_\delta^*+\delta^{-1})^2}.
	\end{equation*}
	Let $t\in [T_{\delta}^{*},T^{**})$, and integrate on $[T_\delta^{*},t]$: for $\delta$ small enough (dependent on $C_{3}$), this gives
	\begin{align*}
	\mathcal{F}(t) & \lesssim e^{-2\mu (t-T_\delta^*)} \mathcal{F}(T_{\delta}^*)+ \delta  \int_{T_{\delta}^{*}}^{t} e^{-2\mu(t-s)}  \frac{ C_{3}^{2}(C_{3}^{2}+1)\d s}{(s-T_\delta^*+\delta^{-1})^2} \\
	&  \lesssim \delta^2 e^{-2\mu (t-T_{\delta}^*)} + \frac{\delta C_{3}^{2}(C_{3}^{2}+1)}{(t-T_\delta^*+\delta^{-1})^2}\lesssim \frac{1}{(t-T_\delta^*+\delta^{-1})^2},
	\end{align*}
	which means~\eqref{est:FT0}. Therefore,  from the coercivity estimate~\eqref{est:coer2}, combining the bound~\eqref{est:AFa} and~\eqref{est:FT0} on $\q F$ and $\q A$ that we just obtained and the boostrap bounds \eqref{est:Boot}, we get
\begin{align*}
\q N^2 & \lesssim \q F + \mu^{-1}\mathcal{A}\lesssim  \frac{1}{(t-T_\delta^*+\delta^{-1})^{2}} +  \frac{\delta^{\frac{1}{4}}C^{2}_{3}(C_{3}^{2}+1)}{(t-T_\delta^*+\delta^{-1})^{2}} \lesssim \frac{1}{(t-T_\delta^*+\delta^{-1})^{2}},
\end{align*}
which strictly improves the estimate~\eqref{est:Boot} of $\mathcal{N}$, upon $\delta$ small enough.

As a consequence of improving the bootstrap assumption on $a$ and $\mathcal{N}$, we conclude that $T^{**}=\infty$. 

Finally, it suffices to bound the geometric parameters. First, recall the equation \eqref{equ:bl} on $\ell$ and $\beta$: proceeding as for $\q F$ (using \eqref{est:Boot} for $t \ge T_\delta^*$), we get that
\[ \forall t \ge T_\delta^*, \quad |\ell(t)| + |\beta(t)| \lesssim \frac{1}{(t-T_\delta^*+1)^2}. \]
Then we consider $z$ and $\theta$: using the above estimate and \eqref{est:Boot}, the equation \eqref{equ:zta} now writes
\begin{align*}
|\dot z| & \le |\dot z - \ell| + |\ell| \lesssim |\ell| + \q N^2 + |a| \q N \lesssim  \frac{1}{(t-T_\delta^*+1)^2}, \\
|\dot \theta| & \le |\dot \theta - \beta | + |\beta| \lesssim |\beta| + \q N^2 + |a| \q N \lesssim  \frac{1}{(t-T_\delta^*+1)^2}.
\end{align*}
This proves that $z(t) \to z_\infty$ as $t \to +\infty$, and that
\[ \forall t \ge T_\delta^*, \quad \left|z(t)-z_{\infty}\right| \lesssim \frac{1}{t-T_\delta^*+1}. \]
From Proposition ~\ref{prop:trapping}, we already know that $\theta(t) \to 0$ as $t \to +\infty$, and we obtain as for $z$ the convergence rate
\[ \forall t \ge T_\delta^*, \quad \left|\theta(t)\right| \lesssim \frac{1}{t-T_\delta^*+1}. \]
Finally, gathering the above estimates, we conclude that for all $t\ge T_{0}$,
	\begin{align*}
	\MoveEqLeft \|u(t)-q(\cdot-z_{\infty})\|_{H^{1}}+\|\partial_{t}u(t)\|_{L^{2}}\\
	&\le \|u(t)-Q(t)\|_{H^{1}}+\|Q(t)-q(\cdot-z_{\infty})\|_{H^{1}}+\|\partial_{t} u(t)\|_{L^{2}} \\
	&\lesssim \mathcal{N}(t)+|a(t)|+|z(t)-z_{\infty}|+|\theta(t)|\lesssim \frac{1}{t-T_\delta^*+1}.
	\end{align*}
	The proof of Theorem~\ref{thm:2} is complete.
	\end{proof}

\section{Proof of Theorem~\ref{thm:3}}
In this section, we prove Theorem~\ref{thm:3}. Let $q$ be a degree-1 excited state.

\begin{proof}[Proof of Theorem~\ref{thm:3}] Let $0<\delta\ll 1$ to be chosen later. 
	Given ${\boldsymbol{\mathfrak a}}^{+}=(\mathfrak a_{k}^{+})_{k=1,\cdots,K}\in \bar{\mathcal{B}}_{\RR^{K}}(\delta^{\frac{3}{2}})$, we consider the solution $\vec{u}$ of~\eqref{equDKG} with initial data
	\begin{equation}\label{def:ini}
	\vec{u}(0)=(q,0)+(\delta\phi,0)+\vec{W}(\mathfrak a_{k}^{+}),
	\end{equation}
	where
	\begin{equation}\label{def:W}
	\vec{W}(\boldsymbol{\mathfrak a}^{+})=\sum_{k=1}^{K} \frac{\mathfrak a_{k}^{+}}{(\zeta_{k}^+)^{2}+1} \begin{pmatrix} \zeta_{k}^{+}Y_{k} \\ Y_{k}\end{pmatrix}.	\end{equation}
	\emph{Step 1. Decomposition.} For any $t\ge 0$ such that $\vec{u}$ is defined and satisfies~\eqref{est:dec}, we
	consider its decomposition according to Proposition~\ref{pro:dec}: this gives the functions $z, \ell, \theta, \beta, a,b$ and $\vec \varphi$. Note that by the definition of $\vec{W}(\boldsymbol{\mathfrak a}^{+})$, the initial data $\vec{u}(0)$ is modulated in the sense that
	\begin{equation}\label{equ:ini1}
	\vec{\varphi}(0)=\vec{W}(\ba),\ z(0)=\ell(0)=0,\ 
	\theta(0)=\beta(0)=0,\ b(0)=0,\ a(0)=\delta.
	\end{equation}
	Also, by.~\eqref{equ:Yk}, and as $\vec Z_k^\pm(0) = \begin{pmatrix} \zeta_k^\pm Y_k \\ Y_k \end{pmatrix}$, we have
	\begin{equation}\label{equ:ini2}
	\forall k =1, \dots, K, \quad a_{k}^{+}(0)=\langle \vec{W}(\boldsymbol{a}), \vec Z_k(0) \rangle= \mathfrak a_{k}^{+}.
	\end{equation}
	 We introduce the following bootstrap setting,
	 \begin{equation}\label{boot}
	 \left\{\begin{aligned}
	 &\left|a(t)-(t+\delta^{-1})^{-1}\right|\le (t+\delta^{-1})^{-\frac{8}{7}}+(t+\delta^{-1})^{-\frac{\bar{p}}{2}},\\
	 &|\theta(t)|\le \delta^{\frac{1}{2}},\quad \mathcal{N}(t)\le (t+\delta^{-1})^{-\frac{5}{4}},\quad 
	 \mathcal{A}(t)\le ( t+\delta^{-1})^{-3}.
	 \end{aligned}
	 \right.
	 \end{equation}
Let $T_{*} = T_*({\boldsymbol{\mathfrak a}}^{+})$ be the supremum of times $t \ge 0$ such that $\vec u$ is defined on $[0,t]$, satisfies the assumption \eqref{est:dec} of Proposition~\ref{pro:dec} on $[0,t]$, and such that the bootstrap estimates \eqref{boot} hold on $[0,t]$.
	
	Our goal is to prove that there exists at least one choice of ${\boldsymbol{\mathfrak a}}^{+} \in \bar{\mathcal{B}}_{\RR^{K}}(\delta^{\frac{3}{2}})$ such that  $T_{*}(\mathfrak a^+)=\infty$. For this, we start by closing all bootstrap estimates except the one for the instable modes, $\q A$, given any ${\boldsymbol{\mathfrak a}}^{+} \in \bar{\mathcal{B}}_{\RR^{K}}(\delta^{\frac{3}{2}})$. Then we prove the existence of suitable parameters $\boldsymbol{a}^{+}=(a_{k}^{+})_{k=1,\cdots,K}$ for which $\q A$ is controlled, using a topological argument. Before we proceed with the actual proof, let us emphasize a major difference between the argument here and the previous boostraps. In Theorem~\ref{thm:3} the goal is \emph{to construct} a solution, and our only choice is a good guess for the initial data as in \eqref{def:ini}; hence, we do not have any long time a priori knowledge, in particular, we have no way to ensure a priori that $\partial_t u \in L^2([0,+\infty),L^2)$ (a bound which plays a key role in controlling the unstable modes): it will be consequence of our construction.
	
	In the remainder of the proof, the implied constants in $\lesssim $ or $O$ do not depend on the small parameter $\delta>0$ which appears in the definition of the initial data \eqref{def:ini} and in the bootstrap assumptions~\eqref{boot}, nor on ${\boldsymbol{\mathfrak a}}^{+} \in \bar{\mathcal{B}}_{\RR^{K}}(\delta^{\frac{3}{2}})$.
	
	\smallskip
	
	\emph{Step 2. Closing the estimates in \eqref{boot} except for $\q A$.}
	
	Fix ${\boldsymbol{\mathfrak a}}^{+} \in \bar{\mathcal{B}}_{\RR^{K}}(\delta^{\frac{3}{2}})$. For simplicity of notations, we drop any references to $\boldsymbol{\mathfrak a}^{+}$ in this step.
	
	\emph{Estimate of $a$.} From~\eqref{boot} and the definition of $\mathcal{R}_{1}$, we have, for any $t \in [0,T_*)$,
	\begin{equation}\label{est:R12a2}
	\left|\mathcal{R}_{1}(t)^2-a(t)^2\right|\lesssim |a(t)|\mathcal{N}(t)+\mathcal{N}(t)^{2}\lesssim (t+\delta^{-1})^{-\frac{9}{4}}.
	\end{equation}
	Also, 
		\begin{equation}\label{est:R1-t} 
	\left|\mathcal{R}_{1}(t)-(t+\delta^{-1})^{-1}\right|\le (t+\delta^{-1})^{-\frac{8}{7}}+(t+\delta^{-1})^{-\frac{\bar{p}}{2}} + (t+\delta^{-1})^{-\frac{5}{4}},
	\end{equation}
	and so $\q R_1(t) >0$ if $\delta>0$ is small enough. By~\eqref{est:R1a2},~\eqref{boot} and~\eqref{est:R12a2},
	\begin{equation*}
	\begin{aligned}
	\left|\frac{\d }{\d t}\mathcal{R}_{1}+\mathcal{R}_{1}^{2}\right|
	&\lesssim \left|\mathcal{R}_{1}^{2}(t)-a^{2}(t)\right|+\mathcal{N}^{2}(t)+|a(t)|\mathcal{N}(t)+|a(t)|^{\bar{p}}\\
	&\lesssim (t+\delta^{-1})^{-\frac{9}{4}}+(t+\delta^{-1})^{-\bar{p}}.
	\end{aligned}
	\end{equation*}
	Therefore, from~\eqref{est:R1-t}, we have 
	\begin{equation*}
	\left|\frac{\d}{\d t}\left(\mathcal{R}_{1}^{-1}\right)(t)-1\right|\lesssim (t+\delta^{-1})^{-\frac{1}{4}}+(t+\delta^{-1})^{-\bar{p}+2}.
	\end{equation*}
	Integrating above estimate on $[0,t]$, and using~\eqref{equ:ini1}, we obtain
	\begin{equation*}
	\left|\mathcal{R}_{1}^{-1}(t)-(t+\delta^{-1})\right|\lesssim (t+\delta^{-1})^{\frac{3}{4}}+(t+\delta^{-1})^{-\bar{p}+3},
	\end{equation*}
	so that in particular $\q R_1(t)^{-1} \ge \frac{1}{2} (t+\delta^{-1})$ (for $\delta$ small enough) and
	\[ 
	\left| \q R_1(t) - \frac{1}{t+\delta^{-1}} \right|  = \frac{\left|\mathcal{R}_{1}^{-1}(t)-(t+\delta^{-1})\right|}{\q R_1(t)^{-1} (t+\delta^{-1})} \lesssim (t+\delta^{-1})^{-\frac{5}{4}}+(t+\delta^{-1})^{-\bar{p}+1}.	
	\]
	From there, we infer
	\begin{align*}
	\left|a(t)-(t+\delta^{-1})^{-1}\right|
	&\lesssim \left|\mathcal{R}_{1}(t)-(t+\delta^{-1})^{-1}\right|+\mathcal{N}(t)\\
	&\lesssim (t+\delta^{-1})^{-\frac{5}{4}}+(t+\delta^{-1})^{-\bar{p}+1},
	\end{align*}
	which strictly improves the estimate~\eqref{boot} of $a$ for $\delta$ small enough (recall that $\bar p >2$).
		
	\emph{Estimate of $\theta$.} We use the evolution equations~\eqref{equ:zta} and \eqref{equ:bl} for $\theta$ and $\beta$, under the bootstrap assumption \eqref{boot}: this gives, for all $t\in [0,T_{*})$,
	\begin{equation*}
	\left|\frac{\d}{\d t}\left(\theta+\frac{\beta}{2\alpha}\right)(t)\right|\lesssim \mathcal{N}^{2}(t)+a^{2}(t)\lesssim (t+\delta^{-1})^{-2}.
	\end{equation*}
	Fix $t \in [0,T_{*})$ and integrate the above estimate on $[0,t]$, using~\eqref{equ:ini1} and~\eqref{boot}: it yields
	\begin{equation*}
	\begin{aligned}
	\left|\theta(t)\right|
	&\lesssim \mathcal{N}(t)+ \left|\theta(t)+\frac{\beta(t)}{2\alpha}\right| \lesssim  (t+\delta^{-1})^{-\frac{5}{4}}+ \int_{0}^{t}(s+\delta^{-1})^{-2}\d s\lesssim \delta^{\frac{5}{4}}+\delta\lesssim \delta,
	\end{aligned}
	\end{equation*}
	which strictly improves the estimate~\eqref{boot} of $\theta$, for $\delta$ small enough.
	
	\emph{Estimate of $\mathcal{N}$.} We first derive a bound on $\q F$ using its evolution equation~\eqref{est:dtF}: together with the boostrap assumption \eqref{boot}, we have 
	\begin{equation*}
	\frac{\d}{\d t}\mathcal{F}+2\mu \mathcal{F}\le C_{0}\left((t+\delta^{-1})^{-\frac{15}{4}}+2(t+\delta^{-1})^{-\frac{13}{4}}\right) \lesssim (t+\delta^{-1})^{-\frac{13}{4}}.
	\end{equation*}
	Fix $t \in [0,T_{*})$ and integrate this on $[0,t]$ and using the initial bounds \eqref{def:W} and \eqref{equ:ini1} to infer that
	\begin{equation*}
	\begin{aligned}
	\mathcal{F}(t)
	&\lesssim e^{-2\mu t}\mathcal{F}(0)+e^{-2\mu t}\int_{0}^{t}e^{2\mu s}(t+\delta^{-1})^{-\frac{13}{4}}\d s\\
	& \lesssim \delta^{3}e^{-2\mu t}+(t+\delta^{-1})^{-\frac{13}{4}}\lesssim (t+\delta^{-1})^{-3}.
	\end{aligned}
	\end{equation*}
	To get a bound on $\q N$, we now recall the coercivity bound \eqref{est:coer2} again and the bootstrap assumption \eqref{boot} on $\q A$, and conclude that
	\begin{equation*}
	 \mathcal{N}^{2}(t)\lesssim \mathcal{F}(t)+\mu^{-1}\mathcal{A}(t)\lesssim (t+\delta^{-1})^{-3},
	\end{equation*}
	which strictly improves the boostrap estimate~\eqref{boot} of $\mathcal{N}$ for $\delta$ small enough.
	
	\smallskip
	
	\emph{Step 3. Transversality and choice of ${\boldsymbol{\mathfrak a}}^{+}$.} 
	Observe that for any time $t$ where the bootstrap bounds \eqref{boot} holds together with the equality $\mathcal{A}(t)=( t+\delta^{-1})^{-3}$, the evolution equation \eqref{est:dA2} on $\q A$ yields the following transversality property:
	\begin{align}
	\MoveEqLeft\frac{\d }{\d t}\left((t +\delta^{-1})^{3}\mathcal{A}(t)\right) \ge\left(\nu_{3}(t+\delta^{-1})^{3}+3( t+\delta^{-1})^{2}\right)\mathcal{A}(t) \nonumber\\
	&\quad -C_{0}(t+\delta^{-1})^{3}\left(\mathcal{N}^{3}(t)+a^{2}(t)\mathcal{N}(t)\right) \nonumber\\
	&\ge \nu_{3}+O\left((t+\delta^{-1})^{-\frac{1}{4}}+(t+\delta^{-1})^{-\frac{3}{4}}+(t+\delta^{-1})^{-1}\right)\ge \frac{1}{2}\nu_{3}, \label{est:transverse}
	\end{align}
for $\delta>0$ small enough. Estimate \eqref{est:transverse} is enough to justify the existence of at least
$K$-uple $\mathfrak{a}^{+}\in \bar{B}_{\mathbb{R}^{K}}(\delta^{\frac{3}{2}})$ such that $T_{*}(\mathfrak a^+)=\infty$.

The proof is by contradiction: assume for its sake that for all $\mathfrak{a}^{+}\in \bar{\mathcal{B}}_{\mathbb{R}^{K}}(\delta^{\frac{3}{2}})$, it holds $T_{*}(\mathfrak a^+)<\infty$. Then, a contradiction follows from the following discussion (see for instance more details in~\cite{CMM} and~\cite[Section 3.1]{CMkg}).

\emph{Continuity of the map $\mathfrak a^+ \mapsto T_{*}(\mathfrak a^+)$.} 

Let $\mathfrak{a}^{+}\in \bar{\mathcal{B}}_{\mathbb{R}^{K}}(\delta^{\frac{3}{2}})$: as $T_{*}(\mathfrak a^+) <+\infty$, then as we improved all other estimates in \eqref{boot} in the previous step, necessarily, the equality $\mathcal{A}(t)=( t+\delta^{-1})^{-3}$ holds for $t = T_{*}(\mathfrak a^+)$, and so \eqref{est:transverse} holds at $t=T_{*}(\mathfrak a^+)$.

By continuity of the flow of \eqref{equDKG}  (and of the modulation technique), the above transversality property implies that the map
\begin{equation*}
\bar{\mathcal{B}}_{\mathbb{R}^{K}}(\delta^{\frac{3}{2}}) \to [0,+\infty), \quad \mathfrak a^+ \mapsto T_{*}(\mathfrak a^+)
\end{equation*}
is continuous and there is instantaneous exit for initial data on the boundary:
\begin{equation*}
T_{*}(\mathfrak a^+) =0\quad \text{for all } \mathfrak a^+\in \mathcal{S}_{\mathbb{R}^{K}}\left(\delta^{\frac{3}{2}}\right).
\end{equation*}

\emph{Construction of a retraction}. As a consequence, the map giving the exit point (on the boundary)
\begin{equation*}
\begin{aligned}
\mathcal{M}:\  \bar{\mathcal{B}}_{\mathbb{R}^{K}}(\delta^{\frac{3}{2}})&\to  {\mathcal{S}}_{\mathbb{R}^{K}}(\delta^{\frac{3}{2}})\\
\mathfrak a^+ &\mapsto \delta^{\frac{3}{2}}\left(T_{*}(\mathfrak a^+)+\delta^{-1}\right)^{\frac{3}{2}} \boldsymbol{a}^{+}\left(T_{*}(\mathfrak a^+)\right)
\end{aligned}
\end{equation*}
is well defined, continuous, and moreover, the restriction of $\mathcal{M}$ to
${\mathcal{S}}_{\mathbb{R}^{K}}(\delta^{\frac{3}{2}})$ is the identity. 

The existence of such a map $\q M$ contradicts Brouwer's no retraction theorem for continuous maps from the ball to
the sphere. We conclude that there exists at least one  $\mathfrak{a}^{+}\in \bar{\mathcal{B}}_{\mathbb{R}^{K}}(\delta^{\frac{3}{2}})$ such that $T_{*}(\mathfrak a^+)=\infty$.

\smallskip

\emph{Step 4. Conclusion.}
At this point, we have proved the existence of $\boldsymbol{\mathfrak a}^{+}\in {\mathcal{B}}_{\mathbb{R}^{K}}(\delta^{\frac{3}{2}})$, associated with a global solution $\vec{u} \in \mathscr C([0,+\infty), H^1 \times L^2)$ of~\eqref{equDKG} with initial data defined in~\eqref{def:ini}, which also can be modulated (in the sense of~\eqref{est:dec}) and satisfies~\eqref{boot} for all $t\in [0,\infty)$. Let us now control the convergence of $z$ and $\theta$: using their evolution equations~\eqref{est:zta} (and \eqref{equ:bl}) under the~\eqref{boot} regime, we see that there exist $z(t) \to z_{\infty}$ and $\theta(t) \to \theta_{\infty}$ as $t \to +\infty$ and moreover
\begin{align}
\MoveEqLeft |z(t)-z^{\infty}|+|\theta(t)-\theta^{\infty}| \le \left|z(t)+\frac{\ell}{2\alpha}(t)-z_{\infty}\right|+\left|\theta(t)+\frac{\beta}{2\alpha}(t)-\theta_{\infty}\right|+ \frac{1}{2\alpha} \mathcal{N}(t) \nonumber \\
& \le \frac{1}{2\alpha} \mathcal{N}(t)+ \int_t^{\infty} \left( \left| \dot z + \frac{\dot \ell}{2 \alpha} \right| +   \left| \dot \theta + \frac{\dot \beta}{2 \alpha} \right| \right) ds \nonumber \\
&\lesssim \mathcal{N}(t)+\int_{t}^{\infty}\left(\mathcal{N}^{2}(s)+|a(s)|\mathcal{N}(s)+|a|^{\bar{p}}(s)\right)\d s \nonumber \\
&\lesssim (t+\delta^{-1})^{-\frac{5}{4}}+(t+\delta^{-1})^{-\frac{3}{2}}+(t+\delta^{-1})^{-\bar{p}+1}. \label{est:zthe}
\end{align}
We now use the invariance of the equation to get a solution to \eqref{equDKG} where $z_\infty=0$ and $\theta_\infty=0$ by setting
\begin{equation*}
\vec{u}_{*}(t,x)=\vec{u}(t,R^{-1}_{\theta_{\infty}}(x+z_{\infty}))\quad \text{for all} \quad (t,x)\in [0,\infty)\times \mathbb{R}^{N}.
\end{equation*}
Then, from~\eqref{boot} and \eqref{est:zthe}, $\vec{u}_{*} \in \mathscr C([0,+\infty), H^1 \times L^2)$ is a solution of~\eqref{equDKG} which enjoys requested properties in the conclusion of Theorem~\ref{thm:3}, and more precisely,
\[ \forall t \ge 0, \quad \| \vec{u}_{*}(t) - (q,0) - (t+\delta^{-1})^{-1} (\phi,0) \|_{H^1 \times L^2}  \lesssim  (t+\delta^{-1})^{-\frac{5}{4}} + (t+\delta^{-1})^{-\bar{p}+1}. \]
The proof of Theorem~\ref{thm:3} is complete.
\end{proof}

\end{document}